\documentclass[12pt]{amsart}

\usepackage{amsmath,color}
\usepackage{amssymb}
\usepackage{euscript}
\usepackage{amscd}

\usepackage{placeins}

\usepackage{tikz}
\usetikzlibrary{snakes,arrows,shapes}

\graphicspath{{images/}}

\newtheorem{thm}{Theorem}%[section]

\newtheorem{lem}[thm]{Lemma}

\newtheorem{prop}[thm]{Proposition}

\theoremstyle{definition}
\newtheorem{rem}[thm]{Remark}
\newtheorem{defn}[thm]{Definition}

\newcommand{\bmat}{\left[ \begin{matrix}}
\newcommand{\emat}{\end{matrix} \right]}

\numberwithin{equation}{section}
\numberwithin{figure}{section}

\newcommand{\Circ}{\mathop{\rm Circ}}
\newcommand{\Tr}{\operatorname{\!{^{\mbox{\scriptsize\sf T}}}}\!}
\newcommand{\script}[1]{\EuScript{#1}}
\newcommand{\trace} {\mbox{\rm tr}}

\newcommand{\Pb}{\mathbf P}
\newcommand{\Qb}{\mathbf Q}
\newcommand{\Sigmab}{\boldsymbol{\Sigma}}
\newcommand{\Sb}{\mathbf S}
\newcommand{\Tb}{\mathbf  T}
\newcommand{\Mb}{\mathbf  M}
\newcommand{\Ab}{\mathbf  A}
\newcommand{\Bb}{\mathbf  B}
\newcommand{\Fb}{\mathbf  F}
\newcommand{\Cb}{\mathbf  C}
\newcommand{\Eb}{\mathbf  E}
\newcommand{\Ib}{\mathbf  I}
\newcommand{\Yb}{\mathbf  Y}
\newcommand{\Xb}{\mathbf  X}
\newcommand{\Gb}{\mathbf  G}
\newcommand{\Wb}{\mathbf  W}
\newcommand{\wb}{\mathbf  w}
\newcommand{\Vb}{\mathbf  V}

\newcommand{\Gammab}{\mathbf  \Gamma}

\newcommand{\E}{{\mathbb E}}

\newcommand{\Zbb}{\mathbb Z}

\newcommand{\yb}{\mathbf  y}
\newcommand{\cb}{\mathbf  c}
\newcommand{\pb}{\mathbf  p}
\newcommand{\qb}{\mathbf  q}
\newcommand{\gb}{\mathbf  g}
\newcommand{\db}{\mathbf  d}
\newcommand{\eb}{\mathbf  e}

\newcommand{\gammab}{\boldsymbol{\gamma}}

\chardef\bslash=`\\ % p. 424, TeXbook
\begin{document}
\title[Modeling of Stationary Periodic Time Series]{Modeling of Stationary Periodic Time Series\\ by 
%Bilateral  and Unilateral 
ARMA Representations}

\author[A. Lindquist]{Anders Lindquist\dag}
\author[G. Picci]{Giorgio Picci\ddag}
%\thanks{2000 {\em Mathematics Subject Classification.} Primary 14R15, 55M35. Secondary 47H10.}
%\thanks{{\em Key words and key phrases.} Jacobian conjecture, Smith theory, classifying spaces.} 
%\thanks{This research was supported in part by  grants from
%AFOSR, NSF, VR, and the G\"oran Gustafsson Foundation.}
\thanks{\dag \ Shanghai Jiao Tong University, Shanghai, China, and Royal Institute of Technology, Stockholm, Sweden}
\thanks{\ddag \ Univerity of Padova, Padova, Italy}

\maketitle

\begin{center}
{\em Dedicated to Boris Teodorovich Polyak\\ on the occasion of his 80th birthday}
\end{center}

\begin{abstract} 
This is a survey of some recent results on the rational circulant covariance extension problem: Given a partial sequence $(c_0,c_1,\dots,c_n)$ of covariance lags $c_k=\E\{y(t+k)\overline{y(t)}\}$ emanating from a stationary periodic process $\{y(t)\}$ with period $2N>2n$, find all possible rational spectral functions of $\{y(t)\}$ of degree at most  $2n$ or, equivalently, all bilateral and unilateral ARMA models of order at most $n$, having this partial covariance sequence. Each representation is obtained as the solution of a pair of dual convex optimization problems. This theory is then reformulated in terms of circulant matrices and the connections to reciprocal processes and the covariance selection problem is explained. Next it is shown how the theory can be extended to the multivariate case. Finally, an application to image processing is presented. 
\end{abstract}

\section{Introduction}

The rational covariance extension problem to determine a rational spectral density given a finite number of covariance lags has been studied in great detail  \cite{Kalman,Gthesis,Georgiou,BLGM1,BGuL,SIGEST,Byrnes-L-97,BEL1,BEL2,PEthesis,Pavon-F-12}, and it can be formulated as a (truncated) trigonometric moment problem with a degree constraint. Among other things, it is the basic problem in partial stochastic realization theory \cite{Byrnes-L-97} and certain Toeplitz matrix completion problems. In particular, it provides a parameterization of the family of (unilateral) autoregressive moving-average (ARMA) models of stationary stochastic processes with the same finite sequence of covariance lags. We also refer the reader to the recent monograph \cite{LPbook}, in which this problem is discussed in the context of stochastic realization theory.
  
 Covariance extension for {\em periodic\/} stochastic processes, on the other hand, leads to matrix completion of Toeplitz matrices with circulant structure and to partial stochastic realizations in the form of {\em bilateral\/} ARMA models 
 \begin{displaymath}
\sum_{k=-n}^n q_k y(t-k) = \sum_{k=-n}^n p_k e(t-k)
\end{displaymath} 
for a stochastic processes $\{y(t)\}$, where $\{e(t)\}$ is the corresponding conjugate process. This connects up to a rich realization theory for reciprocal processes \cite{Krener-86,KFL91,Levy-F-02,Levy-F-K-90}. As we shall see there are also (forward and backward) unilateral ARMA representations for periodic processes. 
 
 In \cite{Carli-FPP}  a maximum-entropy approach to this circulant covariance extension problem was presented, 
providing a procedure for determining the unique bilateral AR model matching the covariance sequence. However, more recently it was discovered that  the circulant covariance extension problem can be recast in the context of the optimization-based theory of moment problems with rational measures developed in \cite{BGuL,SIGEST,BGL1,BEL2,BLN,BLmoments,BLkrein,Georgiou3,Georgiou-L-03} allowing for a complete parameterization of all bilateral ARMA realizations.  This led to a a complete theory for the scalar case \cite{LPcirculant}, which was then extended to to the multivariable case in \cite{LMPcirculantMult}. Also see \cite{RLskewperiodic}  for modifications of this theory to skew periodic processes and \cite{RKfast} for fast numerical procedures. 

The AR theory of \cite{Carli-FPP} has been successfully applied to image processing of textures \cite{Chiuso-F-P-05,Picci-C-08}, and we anticipate an enhancement of such methods by allowing for more general ARMA realizations. 

The present survey paper is to a large extent based on \cite{LPcirculant}, \cite{LMPcirculantMult} and \cite{Carli-FPP}. In Section~\ref{periodicsec} we begin by characterizing stationary periodic processes. In Section~\ref{covextsec}  we formulate  the rational covariance extension problem for periodic processes as a moment problem with atomic measure and present the solution in the context of the convex optimization approach of  \cite{BGuL,SIGEST,BGL1,BEL2,BLN,BLmoments,BLkrein}. These results are then reformulated in terms of circulant matrices in Section~\ref{circulantsec} and interpreted in term of bilateral ARMA models in Section~\ref{bilateralsec}  and in terms of unilateral ARMA models in Section~\ref{unilateralARMAsec}.  In Section~\ref{reciprocalsec} we investigate the connections to reciprocal processes of order $n$ \cite{Carli-FPP} and the covariance selection problem of Dempster \cite{Dempster}.  In Section~\ref{logarithmicsec} we consider the situation when both partial covariance data and logarithmic moment (cepstral) data is available.  To simplify the exposition the theory has so far been developed in the context of scalar processes, but in Section~\ref{multivariatesec} we show how it can be extended to the multivariable case. All of these results are illustrated by examples taken from \cite{LPcirculant} and \cite{LMPcirculantMult}. Section~\ref{imagesec} is devoted to applications in image processing.  
%Finally, in Section~\ref{conclusionsec} we provide conclusions and suggest some open problems. 

% As pointed out in \cite{LPcirculant} the circulant rational covariance extension theory provides a fast approximation procedure for solving the regular rational covariance extension problem, as it is based on fast Fourier transforms (FFT), and in the present paper we shall provide numerical evidence that this  also holds in the multivariable case. We stress that this paper is a first step in establishing a complete theory for the multivariable case. Following \cite{BLN} we confine our ARMA models to those whose transfer function have a matrix fraction representation with a scalar numerator polynomial. 

%an approach to model (see e.g. \cite{Krener,KrenerL,krener, P7}), a special class of Gibbs-Markov random fields which can be described by linear stochastic models of parameter descriptor type. This is a natural non-causal extension of the linear state space models  used in  time series analysis. It turns out that the statistical estimation of reciprocal models leads to a problem of covariance selection of the kind discussed in \cite{Dempster}.  In particular, a covariance selection problem for block-circulant matrices is encountered in the statistical modeling of textures. 

\section{Periodic stationary processes}\label{periodicsec}

Consider a zero-mean  full-rank stationary process $\{y(t)\}$,  in general complex-valued, defined on a finite interval $[-N+1,\,N]$ of the integer line $\Zbb$ and extended to all of $\Zbb$ as a periodic stationary process with  period $2N$ so that 
\begin{equation}
\label{periodic2N}
y(t +2kN) =y(t) 
\end{equation}
almost surely. By stationarity there is a representation
\begin{equation}
\label{ }
y(t)=\int_{-\pi}^\pi e^{it\theta}d\hat{y}(\theta), \quad \text{ where $\E \{|d\hat{y}|^2\}=dF(\theta)$}, 
\end{equation}
(see, e.g., \cite[p. 74]{LPbook}), and therefore
\begin{equation}
\label{F2c}
c_k:= \E\{y(t+k)\overline{y(t)}\} = \int_{-\pi}^\pi e^{ik\theta}dF(\theta). 
\end{equation}
Also, in view of \eqref{periodic2N},
\begin{displaymath}
\int_{-\pi}^\pi e^{it\theta}\left(e^{i2N\theta}-1\right)d\hat{y} =0,
\end{displaymath}
and hence
\begin{displaymath}
\int_{-\pi}^\pi \left|e^{i2N\theta}-1\right|^2dF =0,
\end{displaymath}
which shows that the support of $dF$ must be contained in $\{ k\pi/N; \, k=-N+1, \dots, N\}$. Consequently the spectral density of $\{y(t)\}$ consists of point masses on the discrete unit circle $\mathbb{T}_{2N}:=\{\zeta_{-N+1},\zeta_{-n+2},\dots,\zeta_N\}$, where
\begin{equation}
\label{zetadefn}
\zeta_k=e^{ik\pi/N}. 
\end{equation}
More precisely, define the function
\begin{equation}
\label{Phi}
\Phi(\zeta)=\sum_{k=-N+1}^Nc_k\zeta^{-k}
\end{equation} 
on $\mathbb{T}_{2N}$. This is the discrete Fourier transform (DFT) of the sequence $(c_{-N+1},\dots,c_N)$, which can be recovered by the inverse DFT
\begin{equation}
\label{discretemoments}
c_k=\frac{1}{2N}\sum_{j=-N+1}^N \Phi(\zeta_j)\zeta_j^k =\int_{-\pi}^\pi e^{ik\theta}\Phi(e^{i\theta})d\nu, 
\end{equation}
where  $\nu$ is a step function with steps $\frac{1}{2N}$ at each $\zeta_k$; i.e., 
\begin{equation}
\label{nu}
d\nu(\theta) =\sum_{j=-N+1}^N\delta(e^{i\theta}-\zeta_j)\frac{d\theta}{2N}.
\end{equation}
Consequently, by \eqref{F2c}, $dF(\theta)=\Phi(e^{i\theta})d\nu(\theta)$.  We note in passing  that 
\begin{equation}
\label{delta}
\int_{-\pi}^\pi e^{ik\theta}d\nu(\theta)=\delta_{k0},
\end{equation}
where $\delta_{k0}$ equals one for $k=0$ and zero otherwise. To see this, note that, for $k\ne 0$, 
\begin{displaymath}
\begin{split}
(1-\zeta_k)\int_{-\pi}^\pi e^{ik\theta}d\nu&= \frac{1}{2N} \sum_{j=-N+1 }^{N}\left(\zeta_k^j -\zeta_k^{j+1}\right)\\&=  \frac{1}{2N}\left(\zeta_k^{-N+1}-\zeta_k^{N+1}\right)=0.
\end{split}
\end{displaymath}

Since $\{y(t)\}$ is stationary and full rank, the Toeplitz matrix 
\begin{equation}
\label{Toeplitz}
\Tb_n=\begin{bmatrix} c_0&\bar{c}_1&\bar{c}_2&\cdots&\bar{c}_n\\
				c_1&c_0&\bar{c}_1&\cdots& \bar{c}_{n-1}\\
				c_2&c_1&c_0&\cdots&\bar{c}_{n-2}\\
				\vdots&\vdots&\vdots&\ddots&\vdots\\
				c_n&c_{n-1}&c_{n-2}&\cdots&c_0
 		\end{bmatrix}
\end{equation}
is positive definite for all $n\in\mathbb{Z}$. However, this condition is not sufficient for $c_0,c_1,\dots,c_n$ to be a bona-fide covariance sequence of a periodic process, as can be seen from the following simple example.  Consider a real-valued periodic stationary process $y$ of period  four. Then 
$$\E\left\{\begin{bmatrix}y(1)\\y(2)\\y(3)\\y(4)\end{bmatrix}\begin{bmatrix}y(1)&y(2)&y(3)&y(4)\end{bmatrix}\right\}
=\begin{bmatrix} c_0&c_1&c_2&c_3\\c_1&c_0&c_1&c_2\\c_2&c_1&c_0&c_1\\c_3&c_2&c_1&c_0\end{bmatrix}.$$
Then looking at the covariance matrix for two periods, we obtain
$$E\left\{\begin{bmatrix}y(1)\\y(2)\\\vdots\\y(8)\end{bmatrix}\begin{bmatrix}y(1)&y(2)&\cdots&y(8)\end{bmatrix}\right\}=
\begin{bmatrix} c_0&c_1&c_2&{\color{red}c_3}&c_0&c_1&c_2&{\color{red}c_3}\\
c_1&c_0&c_1&c_2&c_1&c_0&c_1&c_2\\
c_2&c_1&c_0&c_1&c_2&c_1&c_0&c_1\\
{\color{red}c_3}&c_2&c_1&c_0&{\color{red}c_3}&c_2&c_1&c_0\\
c_0&c_1&c_2&{\color{red}c_3}&c_0&c_1&c_2&{\color{red}c_3}\\
c_1&c_0&c_1&c_2&c_1&c_0&c_1&c_2\\
c_2&c_1&c_0&c_1&c_2&c_1&c_0&c_1\\
{\color{red}c_3}&c_2&c_1&c_0&{\color{red}c_3}&c_2&c_1&c_0
\end{bmatrix},$$
which is a Toeplitz matrix only when $c_3=c_1$. Therefore the condition $c_3=c_1$ is necessary. Consequently 
$$\Tb_8=\begin{bmatrix} c_0&c_1&c_2&c_1&c_0&c_1&c_2&c_1\\
c_1&c_0&c_1&c_2&c_1&c_0&c_1&c_2\\
c_2&c_1&c_0&c_1&c_2&c_1&c_0&c_1\\
c_1&c_2&c_1&c_0&c_1&c_2&c_1&c_0\\
c_0&c_1&c_2&c_1&c_0&c_1&c_2&c_1\\
c_1&c_0&c_1&c_2&c_1&c_0&c_1&c_2\\
c_2&c_1&c_0&c_1&c_2&c_1&c_0&c_1\\
c_1&c_2&c_1&c_0&c_1&c_2&c_1&c_0
\end{bmatrix} $$
is a {\em circulant matrix}, where the columns are shifted cyclically, the last component moved to the top. Circulant matrices will play a key role in the following. 

\section{The covariance extension problem for periodic processes}\label{covextsec}

Suppose that we are given a partial covariance sequence $c_0,c_1,\dots,c_n$ with $n< N$ such that the Toeplitz matrix $\Tb_n$ is positive definite. Consider the problem of finding and extension $c_{n+1},c_{n+2},\dots,c_{2N}$ so that the corresponding sequence  $c_0,c_1,\dots,c_N$ is the covariance sequence of a stationary process of period $2N$. 

In general this problem will have infinitely many solutions, and, for reasons that will become clear later, we shall restrict our attention to spectral function \eqref{Phi} which are rational in the sense that 
\begin{equation}
\label{Phi=P/Q}
\Phi(\zeta)=\frac{P(\zeta)}{Q(\zeta)},
\end{equation}
where $P$ and $Q$ are Hermitian pseudo-polynomials of degree at most $n$, that is of the form
\begin{equation}
\label{P}
P(\zeta)=\sum_{k=-n}^n p_k \zeta^{-k}, \quad p_{-k}=\bar{p}_k. 
\end{equation}

Let $\mathfrak{P}_+(N)$ be the cone of all pseudo-polynomials \eqref{P} that are positive on the discrete unit circle $\mathbb{T}_{2N}$, and let $\mathfrak{P}_+\subset\mathfrak{P}_+(N)$ be the subset of pseudo-polynomials \eqref{P} such that $P(e^{i\theta})>0$ for all $\theta\in [-\pi,\pi]$. Moreover let $\mathfrak{C}_+(N)$ be the dual cone of all 
partial covariance sequences $\cb=(c_0,c_1,\dots,c_n)$ such that 
\begin{displaymath}
\langle \cb,\pb\rangle :=\sum_{k=-n}^n c_k\bar{p}_k >0 \quad \text{for all $P\in\overline{\mathfrak{P}_+(N)}\setminus\{0\}$},
\end{displaymath}
and let $\mathfrak{C}_+$ be defined in the same way as the dual cone of $\mathfrak{P}_+$. It can be shown \cite{KreinN} that $\cb\in\mathfrak{C}_+$ is equivalent to the Toeplitz condition $\Tb_n>0$. Since $\mathfrak{P}_+\subset\mathfrak{P}_+(N)$, we have $\mathfrak{C}_+(N)\subset \mathfrak{C}_+$, so in general $\cb\in\mathfrak{C}_+(N)$ is a stricter condition than $\Tb_n>0$. 

The proof of the following theorem can be found in \cite{LPcirculant}. 

\begin{thm}\label{mainthm}
Let $\cb\in\mathfrak{C}_+(N)$. Then, for each $P\in\mathfrak{P}_+(N)$, there is a unique $Q\in\mathfrak{P}_+(N)$ such that 
\begin{displaymath}
\Phi=\frac{P}{Q}
\end{displaymath}
satisfies the moment conditions 
\begin{equation}
\label{momentconditions}
\int_{-\pi}^\pi e^{ik\theta}\Phi(e^{i\theta})d\nu(\theta) =c_k, \quad k=0,1,\dots,n.
\end{equation}
\end{thm}

Consequently the family of solutions \eqref{Phi=P/Q} of the covariance extension problem stated above are parameterized by $P\in\mathfrak{P}_+(N)$ in a bijective fashion. From the following theorem we see that, for any $P\in\mathfrak{P}_+(N)$, the corresponding unique $Q\in\mathfrak{P}_+(N)$ can be obtained by convex optimization. We refer the reader to \cite{LPcirculant} for the proofs. 

\begin{thm}\label{optthm}
Let $\cb\in\mathfrak{C}_+(N)$ and $P\in\mathfrak{P}_+(N)$. Then  the problem to maximize
\begin{equation}
\label{primal}
\mathbb{I}_P(\Phi) =\int_{-\pi}^\pi  P(e^{i\theta})\log \Phi(e^{i\theta})d\nu
\end{equation}
subject to the moment conditions \eqref{momentconditions} has a unique solution, namely \eqref{Phi=P/Q}, where $Q$ is the unique optimal solution of the problem to minimize \
\begin{equation}
\label{dual}
\mathbb{J}_P(Q)= \langle \cb,\qb\rangle -\int_{-\pi}^\pi  P(e^{i\theta})\log Q(e^{i\theta})d\nu
\end{equation}
over all $Q\in\mathfrak{P}_+(N)$, where $\qb:=(q_0,q_1,\dots,q_n)$.  The functional $\mathbb{J}_P$ is strictly convex. 
\end{thm}

Theorems \ref{mainthm} and \ref{optthm} are discrete versions of corresponding results in \cite{BGuL,SIGEST}.
The solution corresponding to $P=1$ is called the {\em maximum-entropy solution\/} by virtue of \eqref{primal}. 

\begin{rem}\label{rem:Phiinfty}
As $N\to\infty$ the process $y$ looses it periodic character, and its spectral density $\Phi_\infty$ becomes continuous and defined on the whole unit circle so that
\begin{equation}
\label{continuousmoments}
\int_{-\pi}^\pi e^{ik\theta}\Phi_\infty(e^{i\theta})\frac{d\theta}{2\pi} =c_k, \quad k=0,1,\dots,n.
\end{equation}
 In fact, denoting by $Q_N$ the solution of Theorem~\ref{mainthm}, it was shown in \cite{LPcirculant} that  $\Phi_\infty=P/Q_\infty$, where, for each fixed $P$,
\begin{displaymath}
Q_\infty=\lim_{N\to\infty}Q_N
\end{displaymath}
is the unique $Q$ such that $\Phi_\infty=P/Q$ satisfies the moment conditions \eqref{continuousmoments}.
\end{rem}

\section{Reformulation in terms of circulant matrices}\label{circulantsec}

Circulant matrices are Toeplitz matrices with a special circulant structure
\begin{equation}
\label{ }
\Circ\{ \gamma_0,\gamma_1,\dots, \gamma_\nu\} =
\begin{bmatrix}\gamma_0&\gamma_\nu&\gamma_{\nu-1}&\cdots&\gamma_1\\
			\gamma_1&\gamma_0&\gamma_\nu&\cdots&\gamma_2\\
			\gamma_2&\gamma_1&\gamma_0&\cdots&\gamma_3\\
			\vdots&\vdots&\vdots&\ddots&\vdots\\
			\gamma_\nu&\gamma_{\nu-1}&\gamma_{\nu-2}&\cdots&\gamma_0
\end{bmatrix},
\end{equation}
where the columns (or, equivalently, rows) are shifted cyclically, and where  $\gamma_0,\gamma_1,\dots,\gamma_\nu$ here are taken to be  complex numbers.  In our present covariance extension problem we consider {\em Hermitian\/} circulant  matrices 
\begin{equation}
\label{M_C}
\Mb:=\Circ\{ m_0,m_1,m_2,\dots, m_N,\bar{m}_{N-1},\dots,\bar{m}_2,\bar{m}_1\},
\end{equation}
which can be represented in form
\begin{equation}
\label{S2C}
\Mb =\sum_{k=-N+1}^N m_k\Sb^{-k}, \quad m_{-k}=\bar{m}_k
\end{equation}
where $\Sb$ is the nonsingular $2N\times 2N$  cyclic shift matrix
\begin{equation}
\label{Sb}
\Sb := \left[\begin{array}{cccccc}0 & 1 & 0 & 0 & \dots & 0 \\0 & 0 & 1 & 0 & \dots & 0 \\0 & 0 & 0 & 1 & \dots & 0 \\\vdots & \vdots & \vdots & \ddots & \ddots & \vdots \\0 & 0 & 0 & 0 & 0 & 1 \\1 & 0 & 0 & 0 & 0 & 0\end{array}\right].
\end{equation}
The pseudo-polynomial
\begin{equation}
\label{symbol}
M(\zeta)=\sum_{k=-N+1}^N m_k \zeta^{-k}, \quad m_{-k}=\bar{m}_k
\end{equation}
is called the {\em symbol\/} of $\Mb$. Clearly $\Sb$ is itself a circulant matrix (although not Hermitian) with symbol $S(\zeta)= \zeta$. A necessary and sufficient condition for a matrix $\Mb$ to be circulant is that 
\begin{equation}
\label{ }
\Sb\Mb\Sb\Tr=\Mb.
\end{equation}
Hence, since $\Sb^{-1}=\Sb\Tr$, the inverse of a circulant matrix is also circulant.  More generally, if $\Ab$ and $\Bb$ are circulant matrices of the same dimension with symbols $A(\zeta)$ and $B(\zeta)$ respectively, then $\Ab\Bb$ and $\Ab+\Bb$ are circulant matrices with symbols $A(\zeta)B(\zeta)$ and $A(\zeta)+B(\zeta)$, respectively. In fact, the circulant matrices of a fixed dimension form an algebra -- more precisely, a commutative *-algebra with the involution * being the conjugate transpose -- and the DFT is an {\em algebra   homomorphism} of the set of  circulant matrices  onto the   pseudo-polynomials of degree at most $N$ in the variable $\zeta \in \mathbb{T}_{2N}$. Consequently, circulant matrices commute, and, if $\Mb$ is a circulant matrix with symbol $M(\zeta)$ then $\Mb^{-1}$ is circulant with symbol $M(\zeta)^{-1}$.

The proof of the following proposition is immediate.

\begin{prop}\label{prop:Sigma}
Let $\{y(t); \, t=-N+1,\dots,N\}$ be a stationary process with period $2N$ and covariance lags \eqref{F2c}, and let $\yb$ be the $2N$-dimensional stochastic vector $\yb=[y(-N+1),y(-N+2), \cdots, y(N)]\Tr$. Then, with  $^*$ denoting conjugate transpose,
\begin{equation}
\label{Sigma}
\Sigmab :=\E\{\yb\yb^*\} =\Circ\{ c_0,c_1,c_2,\dots, c_N,\bar{c}_{N-1},\dots,\bar{c}_2,\bar{c}_1\}
\end{equation}  
is a $2N\times 2N$ Hermitian circulant matrix with symbol $\Phi(\zeta)$ given by \eqref{Phi}.
\end{prop}

The covariance extension problem of Section~\ref{covextsec}, called the {\em circulant rational covariance extension problem}, can now be reformulated as a matrix extension problem. The given covariance data $\cb=(c_0,c_1,\dots,c_n)$ can be represented as a circulant matrix
\begin{equation}
\label{bandedC}
\Cb = \Circ\{ c_0,c_1,\dots,c_n,0,\dots,0,\bar{c}_n,\bar{c}_{n-1},\dots,\bar{c}_1\}
\end{equation}
with symbol
\begin{equation}
\label{C(z)}
C(\zeta)=\sum_{k=-n}^n c_k \zeta^{-k}, 
\end{equation}
where the unknown covariance lags  $c_{n+1},c_{n+2},\dots,c_N$ in \eqref{Sigma}, to be determined, here are replaced by  zeros. A circulant matrix of type \eqref{bandedC} is called {\em banded of order\/} $n$. We recall that $n<N$. From now one we drop the attribute `Hermitian'  since we shall only consider such circulant matrices in the sequel. A banded circulant matrix of order $n$ will thus be determined by $n+1$ (complex) parameters.

The next lemma establishes the connection between circulant matrices and their symbols. 

\begin{lem}\label{diagonalizationlem}
Let $\Mb$ be a circulant matrix with symbol $M(\zeta)$. Then 
\begin{equation}
\label{Mdiag}
\Mb=\Fb^*\text{\rm diag}\big(M(\zeta_{-N+1}),M(\zeta_{-N+2}),\dots,M(\zeta_N)\big)\Fb,
\end{equation}
where $\Fb$ is the unitary matrix
\begin{equation}
\label{F}
\Fb= \frac{1}{\sqrt{2N}}\left[\begin{array}{cccc}\zeta_{-N+1}^{N-1} & \zeta_{-N+1}^{N-2} & \cdots & \zeta_{-N+1}^{-N} \\ \vdots & \vdots & \cdots & \vdots \\\zeta_{0}^{N-1} & \zeta_{0}^{N-2} & \cdots & \zeta_{0}^{-N}  \\ \vdots & \vdots & \cdots & \vdots \\\zeta_{N}^{N-1} & \zeta_{N}^{N-2} & \cdots & \zeta_{N}^{-N} \end{array}\right] .
\end{equation}
Moreover, if $M(\zeta_k)>0$ for all $k$, then
\begin{equation}
\label{logM}
\log\Mb=\Fb^*\text{\rm diag}\big(\log M(\zeta_{-N+1}),\log M(\zeta_{-N+2}),\dots,\log M(\zeta_N)\big)\Fb.
\end{equation}
\end{lem}

\begin{proof}
The discrete Fourier transform $\script F$ maps a sequence   $(g_{-N+1},g_{-N+2},\dots,g_N)$  into the  sequence  of complex numbers 
\begin{equation} \label{DFT}
G(\zeta_{j}) := \sum_{k=-N+1 }^N g_k \zeta_{j}^{-k} \,,\qquad j=-N+1,-N+2, \ldots , N .
\end{equation}
The sequence $\gb$ can be recovered from  $G$  by the inverse transform
\begin{equation} \label{InvDFTmeas}
g_k = \int_{-\pi}^\pi e^{ik\theta}G(e^{i\theta})  d\nu(\theta),\quad k =  -N+1,-N+2, \dots ,N.
\end{equation}
This correspondence can be written 
\begin{equation}
\label{ghat=Fg}
\hat{\gb} =\Fb\gb,
\end{equation}
where $\hat{\gb}:=(2N)^{-\frac12}\big(G(\zeta_{-N+1}),G(\zeta_{-N+2}),\dots,G(\zeta_N)\big)\Tr$, $\gb:=(g_{-N+1},g_{-N+2},\dots,g_N)\Tr$, and $\Fb$ is the nonsingular $2N\times 2N$ Vandermonde matrix \eqref{F}. Clearly $\Fb$ is unitary. Since 
\begin{displaymath}
\Mb\gb= \sum_{k=-N+1 }^N m_k\Sb^{-k}
\end{displaymath}
and $[\Sb^{-k}\gb]_j =g_{j-k}$, where $g_{k+2N}=g_k$, we have 
\begin{displaymath}
\begin{split}
\script F(\Mb\gb)&=\sum_{j=-N+1}^N \zeta^{-j}\sum_{k=-N+1}^N m_kg_{j-k}\\
			&=\sum_{k=-N+1}^N m_k \zeta^{-k}\sum_{j=-N+1}^N g_{j-k}\zeta^{-(j-k)}
			= M(\zeta)\script F\gb,
\end{split}
\end{displaymath}
which yields 
\begin{displaymath}
\sqrt{2N}(\Fb\Mb\gb)_j=M(\zeta_j)\sqrt{2N}(\Fb\gb)_j, \quad j=-N+1,-N+2,\dots,N,
\end{displaymath}
from which \eqref{Mdiag} follows. Finally, since $\log M(\zeta)$ is analytic in the neighborhood of each $M(\zeta_k)>0$,  the eigenvalues of $\log\Mb$ are just the real numbers $\log M(\zeta_k)$, $k=-N+1,\dots,N$, by the spectral mapping theorem \cite[p. 557]{DS}, and hence \eqref{logM} follows. 
\end{proof}

We are now in a position to reformulate Theorems \ref{mainthm} and \ref{optthm} in terms of circulant matrices. To this end first note that, in view of Lemma~\ref{diagonalizationlem},  the cone $\mathfrak{P}_+(N)$ corresponds to the class of positive-definite banded $2N\times 2N$ circulant matrices $\Pb$ of order $n$. Moreover, by Plancherel's Theorem for DFT, which is a simple consequence of \eqref{delta}, we have 
\begin{displaymath}
\sum_{k=-n}^n c_k\bar{p}_k =\frac{1}{2N}\sum_{j=-N+1}^N C(\zeta_j)P(\zeta_{j}),
\end{displaymath}
and hence, by Lemma~\ref{diagonalizationlem}, 
\begin{equation}
\label{ }
\langle \cb,\pb\rangle  = \frac{1}{2N} \trace(\Cb\Pb).
\end{equation}
Consequently, $\cb\in\mathfrak{C}_+(N)$ if and only if $\trace(\Cb\Pb)>0$ for all nonzero, positive-semidefinite, 
banded $2N\times 2N$ circulant matrices $\Pb$ of order $n$. Moreover, if $\Qb$ and $\Pb$ are circulant matrices with symbols $P(\zeta)$ and $Q(\zeta)$, respectively, then, by Lemma~\ref{diagonalizationlem}, $P(\zeta)/Q(\zeta)$ is the symbol of $\Qb^{-1}\Pb$. Therefore Theorem~\ref{mainthm} has the following matrix version. 

\begin{thm}\label{mainthm_matrix}
Let $\cb\in\mathfrak{C}_+(N)$, and let $\Cb$ be the corresponding circulant matrix \eqref{bandedC}. Then, for each  
positive-definite banded $2N\times 2N$ circulant matrices $\Pb$ of order $n$, there is unique positive-definite banded $2N\times 2N$ circulant matrices $\Qb$ of order $n$ such that 
\begin{equation}
\label{Sigmab=QbinvPb}
\Sigmab=\Qb^{-1}\Pb
\end{equation}
is a circulant extension \eqref{Sigma} of $\Cb$.
\end{thm}

In the same way, Theorem~\ref{optthm} has the following matrix version, as can be seen by applying Lemma~\ref{diagonalizationlem}.

\begin{thm}\label{optthm_matrix}
Let $\cb\in\mathfrak{C}_+(N)$, and let $\Cb$ be the corresponding circulant matrix \eqref{bandedC}. Moreover, let $\Pb$ be a positive-definite banded $2N\times 2N$ circulant matrix of order $n$. Then  the problem to maximize
\begin{equation}
\label{primal_matrix}
\mathcal{I}_{\Pb}(\Sigmab) = \text{\rm trace}(\Pb\log\Sigmab)
\end{equation}
subject to 
\begin{equation}
\label{momentcondmatrixb}
\Eb_n\Tr\Sigmab \Eb_n =\Tb_n, \quad \text{where\;} \Eb_n =\begin{bmatrix}\Ib_n\\{\bold 0}\end{bmatrix}
\end{equation}
has a unique solution, namely \eqref{Sigmab=QbinvPb}, where $\Qb$ is the unique optimal solution of the problem to minimize
\begin{equation}
\label{dual_matrix}
\mathcal{J}_{\Pb}(\qb)= \text{\rm trace}(\Cb\Qb) - \text{\rm trace}(\Pb\log\Qb)
\end{equation}
over all positive-definite banded $2N\times 2N$ circulant matrices $\Qb$ of order $n$, where $\qb:=(q_0,q_1,\dots,q_n)$.  The functional $\mathcal{J}_{\Pb}$ is strictly convex. 
\end{thm}

\section{Bilateral ARMA models}\label{bilateralsec}

Suppose now that we have determined a circulant matrix extension \eqref{Sigmab=QbinvPb}. Then there is a stochastic vector $\yb$ formed from the a stationary periodic process with corresponding covariance lags \eqref{F2c} so that 
\begin{displaymath}
\Sigmab :=\E\{\yb\yb^*\} =\Circ\{ c_0,c_1,c_2,\dots, c_N,\bar{c}_{N-1},\dots,\bar{c}_2,\bar{c}_1\}.
\end{displaymath}
 Let $\hat{\E}\{y(t)\mid  y(s),\, s \neq t\}$ be the wide sense conditional mean of $y(t)$ given all $\{y(s),\, s \neq t\}$. Then the error process
\begin{equation}  \label{finconjn}
        d(t)     := y(t)-\hat{\E} \{ y(t)\mid y(s),\, s \neq t\}
  \end{equation}
is orthogonal to all random variables $\{ y(s),\, s \neq t\}$, i.e., $\E\{y(t)\,\overline{d(s)}\}= \sigma^2 \, \delta_{ts}$, $t, s \in\Zbb_{2N}:=\{-N+1,-N+2,\dots,N\}$, where $\sigma^2$ is a  positive number. Equivalently, $\E\{\yb\db^*\}=\sigma^2 \Ib$, where $\Ib$ is the $2N\times 2N$ identity matrix. Setting $\eb:=\db/\sigma^2$, we then have 
\begin{equation}
\label{eystar}
\E\{\eb\yb^*\}=\Ib,
\end{equation}
i.e., the corresponding process $e$ is the {\em conjugate process\/} of $y$ \cite{Masani-60}. Interpreting \eqref{finconjn} in the $\mod 2N$ arithmetics of $\Zbb_{2N}$,    $\yb$ admits a linear representation of the form 
\begin{equation}
\label{Ay=e}
\Gb\yb= \eb, 
\end{equation}
where  $\Gb$ is a $2N\times 2N$ Hermitian circulant matrix  with ones on the main diagonal. Since $\Gb\E\{\yb\yb^*\}=  \E\{\eb\yb^*\} =  \Ib$, $\Gb$ is also positive definite and the covariance matrix $\Sigmab$ is  given by
\begin{equation}
\label{cov}
\Sigmab   = \Gb ^{-1},
\end{equation}
which is circulant, since the inverse of a circulant matrix is itself circulant. In fact, a  stationary process $\yb$  is full-rank periodic in $\Zbb_{2N}$, if and only if $\Sigmab $ is a Hermitian  positive  definite circulant matrix \cite{Carli-FPP}. 

Since $\Gb$ is a Hermitian circulant matrix, it has a symbol
\begin{displaymath}
G(\zeta)=\sum_{k=-N+1}^N g_k\zeta^{-k}, \quad g_{-k}=\bar{g}_k,
\end{displaymath} 
and the linear equation can be written in the autoregressive (AR) form
\begin{equation}
\label{AR}
\sum_{k=-N+1}^N g_k y(t-k)=e(t).
\end{equation}
However,  in general  $\Gb$ is not banded  and $n<<N$, and therefore \eqref{AR} is not a useful representation. Instead using the solution  \eqref{Sigmab=QbinvPb}, we have $\Gb=\Pb^{-1}\Qb$, where $\Pb$ and $\Qb$ are banded of order $n$ with symbols
\begin{displaymath}
P(\zeta)=\sum_{k=-n}^n p_k\zeta^{-k}\quad \text{and}\quad  Q(\zeta)=\sum_{k=-n}^n q_k\zeta^{-k},
\end{displaymath}
and hence \eqref{Ay=e} can be written
\begin{displaymath}
\Qb\yb=\Pb\eb,
\end{displaymath}
or equivalently in the ARMA form
\begin{equation}
\label{ARMA}
\sum_{k=-n}^n q_k y(t-k) = \sum_{k=-n}^n p_k e(t-k).
\end{equation}

Consequently, by Theorem~\ref{mainthm_matrix}, there is a unique bilateral ARMA model \eqref{ARMA} for each banded positive-definite Hermitian circulant matrix $\Pb$ of order $n$, provided $\cb\in\mathfrak{C}_+$. Of course,we could use the maximum-entropy solution with $\Pb=\Ib$ leading to an AR model
\begin{equation}
\label{ME-AR}
\sum_{k=-n}^n q_k y(t-k) = e(t).
\end{equation}

Next, to illustrate the accuracy of bilateral AR modeling by the methods described so far we give some simulations from \cite{LPcirculant}, provided by Chiara Masiero. Given an AR model of order $n=8$ with poles as depicted in Figure \ref{figure1}, we compute a covariance sequence $\cb=(c_0,c_1,\dots,c_n)$ with $n=8$, which is then used to solve the  optimization problem \eqref{dual_matrix} with $\Pb=\Ib$ to obtain a bilateral AR approximations of degree eight for various choices of $N$. In Figure~\ref{figures2and3}, the left picture depicts the spectral density for $N=128$  together with the true spectral density (dashed line), and the right picture illustrates how the estimation error decreases with increasing $N$. 
\begin{figure}[h]
\centering
\includegraphics[height=6.5cm,width=6.5cm]{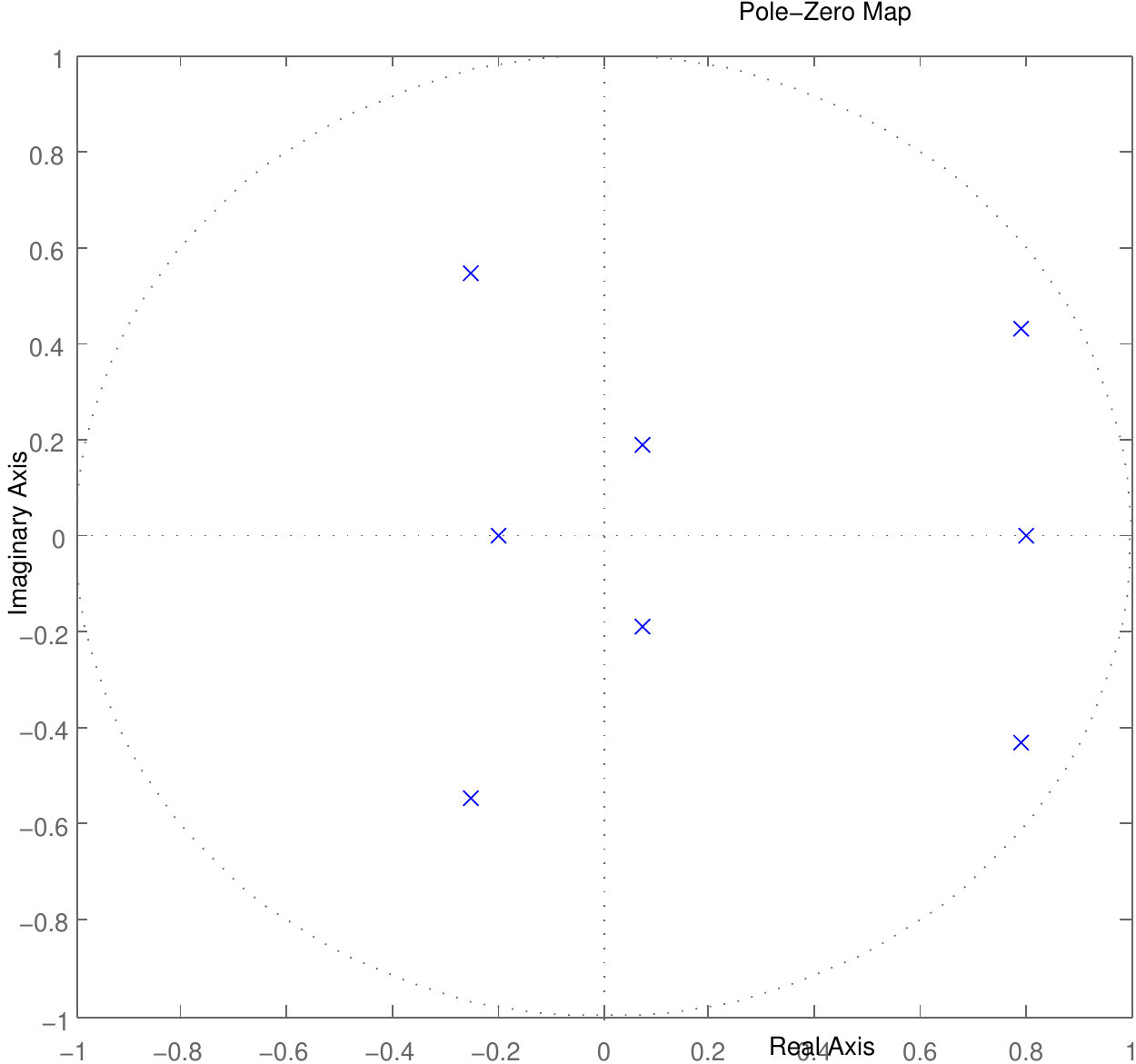}
\caption{\small Poles of true AR model.}
\label{figure1}
\end{figure}
\begin{figure}[bht]
\centering
\includegraphics[height=5cm]{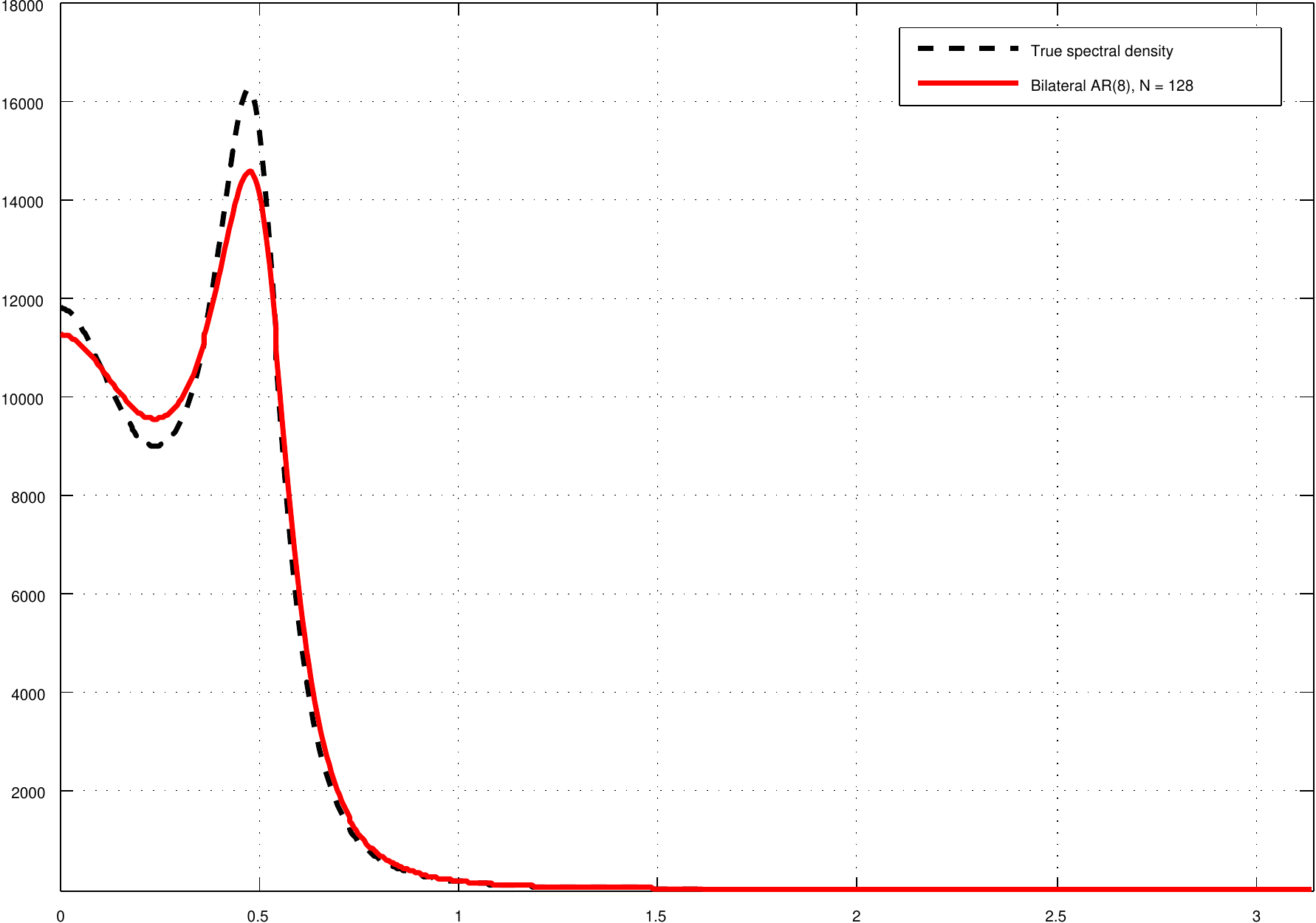}
\label{figure2}
\qquad
\includegraphics[height=5cm]{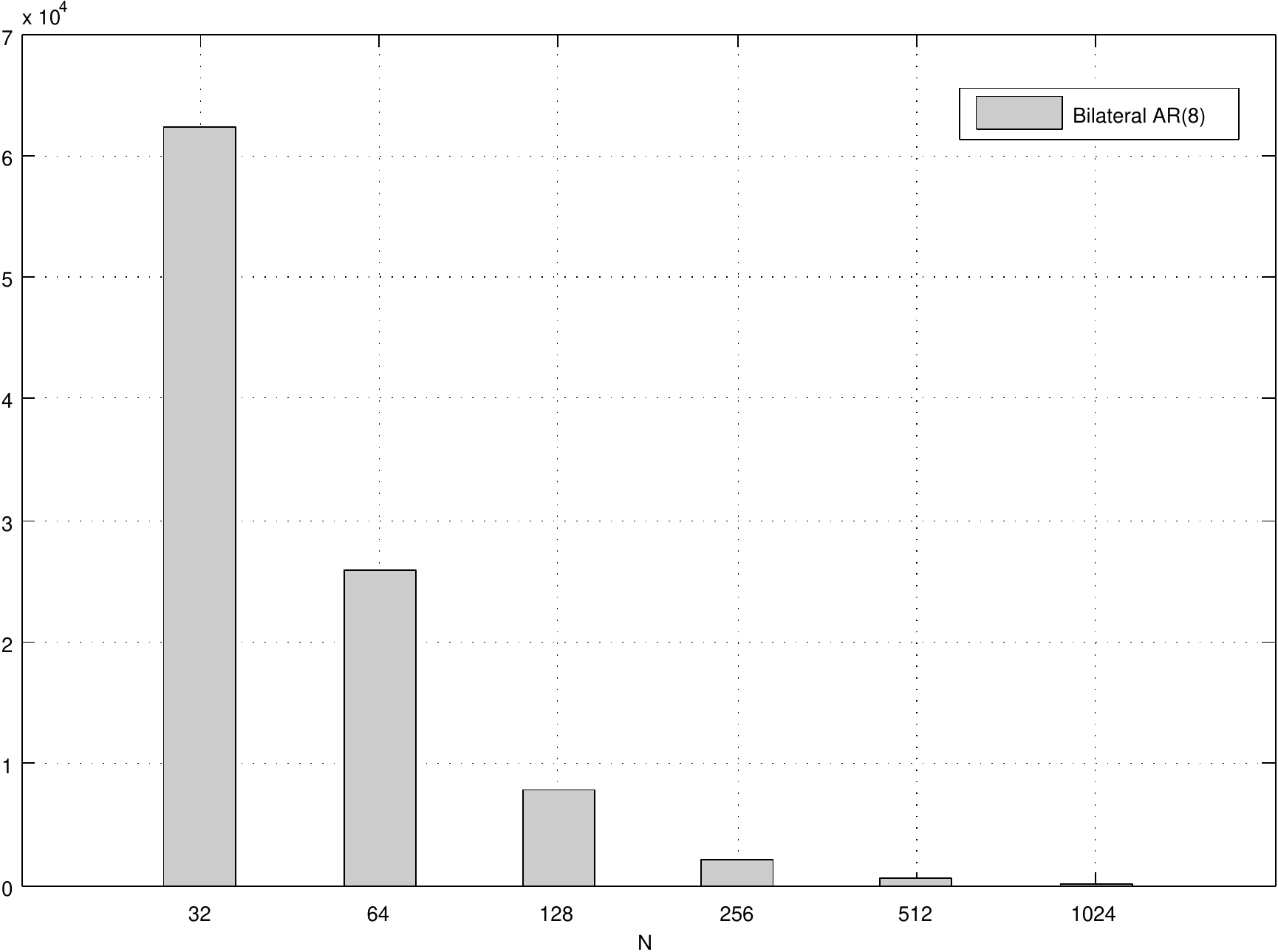}
\caption{\small Bilateral AR approximation: (left) spectrum for N = 128 and true spectrum (dashed); (right) errors for N=32, 64, 128, 256, 512 and 1024.}
\label{figures2and3}
\end{figure}

\section{Unilateral ARMA models and spectral factorization}\label{unilateralARMAsec}

As explained in Section~\ref{periodicsec}, a periodic process $y$ has a discrete spectrum, and Theorem~\ref{mainthm}  provides values of 
\begin{displaymath}
\Phi(z)=\frac{P(z)}{Q(z)}
\end{displaymath}
only in the discrete points $z\in\mathbb{T}_{2N}:=\{\zeta_{-N+1},\zeta_{-n+2},\dots,\zeta_N\}$. Since $\Phi$ takes positive values on $\mathbb{T}_{2N}$, there is a trivial discrete factorization
\begin{equation}
\label{discretefactor}
\Phi(\zeta_k)=W(\zeta_k)W(\zeta_k)^*\quad k=-N+1,\dots, N.
\end{equation}
Defining
\begin{displaymath}
W_k= \frac{1}{2N}\sum_{j=-N+1}^N W(\zeta_j)\zeta_j^k, \quad k=-N+1,\dots, N,
\end{displaymath}
we can write \eqref{discretefactor} in the form 
\begin{equation}
\label{Phi(zeta)}
\Phi(\zeta)=W(\zeta)W(\zeta)^*.
\end{equation}
where $W(\zeta)$ is the discrete Fourier transform
\begin{displaymath}
W(\zeta)=\sum_{k=-N+1}^N W_k\zeta^{-k}.
\end{displaymath}

Applying Lagrange interpolation to $W$, we obtain a spectral factorization equation 
\begin{equation}
\label{Phi(z)}
\tilde\Phi(z)=W(z)W(z)^*,\quad z\in\mathbb{T},
\end{equation}
defined on the whole unit circle, where $\tilde\Phi(\zeta)=\Phi(\zeta)$ on $\mathbb{T}_{2N}$. This is a spectral density of non-periodic stationary process but should not be confused with $\Phi_\infty$ in Remark~\ref{rem:Phiinfty}, which is the unique continuous $\Phi$ with numerator polynomial $P$ and the same covariance structure as the periodic process $y$, i.e.,
\begin{displaymath}
\int_{-\pi}^\pi e^{ik\theta}\Phi_\infty(e^{i\theta})\frac{d\theta}{2\pi} =c_k, \quad k=0,1,\dots,n.
\end{displaymath} 
In fact, although
\begin{equation}
\label{Phitilde2c}
\int_{-\pi}^\pi e^{ik\theta}\tilde\Phi(e^{i\theta})d\nu(\theta) =c_k, \quad k=0,1,\dots,n ,
\end{equation}
the non-periodic process with spectral density $\tilde\Phi$ has the covariance lags
\begin{displaymath}
\tilde{c}_k=\int_{-\pi}^\pi e^{ik\theta}\tilde\Phi(e^{i\theta})\frac{d\theta}{2\pi}, \quad k=0,1,\dots,n ,
\end{displaymath}
which differ from $c_0,c_1,\dots,c_n$. However, setting $\Delta\theta_j:=\theta_j-\theta_{j-1}$ where $e^{\theta_j}=\zeta_j$,   we see from \eqref{zetadefn} that $\Delta\theta_j=\pi/N$ and that the integral \eqref{Phitilde2c} with $\tilde\Phi$ fixed is the Riemann sum
\begin{displaymath}
\sum_{j=-N+1}^Ne^{ik\theta_j} \tilde\Phi(\zeta_j)\frac{\Delta\theta_j}{2\pi} 
\end{displaymath}
converging to $\tilde{c}_k$ for $k=0,1,\dots,n$ as $N\to\infty$. 

By Proposition~\ref{prop:Sigma}, $\Phi(\zeta)$ is the symbol of the circulant covariance matrix $\Sigmab$, and hence \eqref{Phi(zeta)} can be written in the matrix form
\begin{equation}
\label{matrixWW*}
\Sigmab=\Wb\Wb^*,
\end{equation}
where $\Wb$ is the circulant matrix with symbol $W(\zeta)$. The spectral-factorization \eqref{Phi(z)} has a unique outer spectral factor $W(z)$; see, e.g., \cite{LPbook}. As explained in detail in \cite{Carlietal}, this corresponds in the discrete setting to $W(\zeta)$ taking the form
\begin{equation}
\label{analyticW}
W(\zeta)=\sum_{k=0}^N W_k\zeta^{-k},
\end{equation}
which in turn corresponds to $\Wb$ being  {\em lower-triangular circulant}, i.e.,
\begin{equation}
\label{ }
\Wb=\Circ\{ W_0,W_1,\dots,W_N,0,\dots,0\}.
\end{equation}
Note that a lower-triangular circulant matrix is not lower triangular as the circulant structure has to be preserved. 
%In fact, if the bandwidth is $N$ the matrix will be full. 
Since $\Sigmab$ is invertible, then so is $\Wb$.   

Next define the periodic stochastic process $\{w(t),\, t=-N+1\dots,N\}$ for which $\wb = [w(-N +1),w(-N +2),\dots,w(N)]\Tr$ is given by
\begin{equation}
\label{wdefn}
\wb=\Wb^{-1}\yb.
\end{equation}
Then, in view of \eqref{matrixWW*}, we obtain $\E\{\wb\wb^*\}=\Ib$, i.e., the process $w$ is a white noise process. Consequently we have the unilateral representation 
\begin{displaymath}
y(t)=\sum_{k=0}^N W_kw(t-k)
\end{displaymath}
in terms of white noise. 

To construct an ARMA model we appeal to the following result, which is easy to verify in terms of symbols but, as demonstrated in  \cite{Carlietal}, also holds for block circulant matrices considered in Section~\ref{multivariatesec}. 

\begin{lem}\label{lem:bandedfactor}
A positive definite, Hermitian, circulant matrix $\Mb$ admits a factorization $\Mb=\Vb\Vb^*$, where $\Vb$ is of a banded lower-diagonal circulant matrix of order $n<N$, if and only if $\Mb$ is bilaterally banded of order $n$.
\end{lem}

By Theorem~\ref{mainthm_matrix}, $\Sigmab=\Qb^{-1}\Pb$, where $\Qb$ and $\Pb$ are banded, positive definite, Hermitian, circulant matrices of order $n$. 
Hence, by Lemma~\ref{lem:bandedfactor},  there are factorizations
\begin{displaymath}
\Qb=\Ab\Ab^*\quad\text{and}\quad \Pb=\Bb\Bb^*,
\end{displaymath}
where $\Ab$ and $\Bb$ are banded lower-diagonal circulant matrices of order $n$. Consequently, $\Sigmab=\Ab^{-1}\Bb(\Ab^{-1}\Bb)^*$, i.e., 
\begin{equation}
\label{W=AinvB}
\Wb=\Ab^{-1}\Bb,
\end{equation}
which together with \eqref{wdefn} yields $\Ab\yb=\Bb\wb$, i.e., the unilateral ARMA model
\begin{equation}
\label{unilateralARMA}
\sum_{k=0}^n a_k y(t-k) = \sum_{k=0}^n b_k w(t-k).
\end{equation}
Since $\Ab$ is nonsingular, $a_0\ne 0$, and hence we can normalize by setting $a_0=1$.
In particular, if $\Pb=\Ib$, we obtain the AR representation
\begin{equation}
\label{unilateralAR}
\sum_{k=0}^n a_k y(t-k) = b_0 w(t).
\end{equation}

Symmetrically, there is factorization
\begin{equation}
\label{matrixWbarWbar*}
\Sigmab=\bar{\Wb}\bar{\Wb}^*,
\end{equation}
where $\bar{\Wb}$ is upper-diagonal circulant, i.e. the transpose of a lower-diagonal circulant matrix, and a white-noise process
\begin{equation}
\label{wbardefn}
\bar\wb=\bar{\Wb}^{-1}\yb.
\end{equation}
Likewise there are factorizations 
\begin{displaymath}
\Qb=\bar\Ab\bar\Ab^*\quad\text{and}\quad \Pb=\bar\Bb\bar\Bb^*,
\end{displaymath}
where $\bar\Ab$ and $\bar\Bb$ are banded upper-diagonal circulant matrices of order $n$. This yields a backward unilateral ARMA model
\begin{equation}
\label{unilateralARMAbar}
\sum_{k=-n}^0 \bar a_k y(t-k) = \sum_{k=-n}^0 \bar b_k \bar w(t-k).
\end{equation}

These representations should be useful in the smoothing problem for periodic systems \cite{Levy-F-K-90}. 

\section{Reciprocal processes and the covariance selection problem}\label{reciprocalsec}

 Let $\Ab$, $\Bb$ and $\Xb$ be subspaces  in a certain common ambient Hilbert space of zero mean second order random variables. 
% and let $\E\{\cdot\mid\Xb\}$ denote the orthogonal projection onto the subspace $\Xb$ and the vector of componentwise orthogonal projections when applied to a stochastic vector.   
 We say that 
$\Ab$ and $\Bb$ are {\em conditionally orthogonal\/} given $\Xb$ if 
 \begin{equation}
\label{condorthPerp}
\alpha -\hat\E\{\alpha\mid\Xb\}\perp \beta -\hat\E\{\beta\mid\Xb\}, \quad \forall\alpha\in\Ab, \forall\beta\in\Bb 
\end{equation}
(see, e.g.,  \cite{LPbook}), which we denote $\Ab\perp\Bb\mid\Xb$, and which clearly is equivalent to 
\begin{equation}
\label{condorthPerp2}
\E\left\{\hat\E\{\alpha\mid\Xb\}\overline{\hat\E\{\beta\mid\Xb\}}\right\}=\E\{\alpha\overline{\beta}\}, \quad \forall\alpha\in\Ab, \forall\beta\in\Bb.
\end{equation}
Conditional orthogonality is the same as conditional uncorrelatedness, and hence conditional independence  in the Gaussian case.

Let $\yb_{[t-n,t)}$  and $\yb_{(t ,t+n]}$ be the $n$-dimensional random column vectors obtained by stacking $y(t-n),y(t-n+1) \ldots, y(t-1)$  and $y(t+1),y(t+2) \ldots, y(t+n)$, respectively, in that order. In the same way,  $\yb_{[t-n,t]}$ is obtained by appending  $y(t)$ to $\yb_{[t-n,t)}$ as the last element, etc. Here and in the following the sums $t-k$  and $t+k$   are to be understood modulo $2N$. For any interval $(t_1,t_2)\subset [-N+1,N]$, we denote by $(t_1,t_2)^c$ the complementary set in $[1,2N]$.

\begin{defn}\label{def:Recn}
A {\em   reciprocal process of order $n$}  on $ (-N,N]$ is a process $\{y(t); \, t=-N+1,\dots,N\}$ such that 
\begin{equation}
\label{Rec}
\hat\E\{y(t)\mid  y(s),\, s \neq t\}  = \hat\E\{y(t)\mid \yb_{[t-n,t)}\vee\yb_{(t ,t+n]}\}
\end{equation}
for $t\in (-N,N]$.
%the components of $\yb_{(t_1,t_2)}$ are conditionally orthogonal to  the  component of $\yb_{(t_1,t_2)^c}$ given  the space spanned by the components of the $2n$ boundary values  $\yb_{(t_1-n,t_1]}$ and $\yb_{[t_2 ,t_2+n)}$.
\end{defn}

This is a generalization introduced in \cite{Carli-FPP} of the concept of {\em reciprocal process}, which can be trivially extended to vector processes. In fact, a reciprocal process in the original sense is here a reciprocal process of order one. This concept does not require stationarity, although in the paper it will always be assumed.

It follows from \cite[Proposition 2.4.2 (iii)]{LPbook} that $\{y(t)\}$ is reciprocal of order $n$ if and only if 
\begin{equation}\label{RecMod}
\hat\E\{y(t)\mid  y(s),\, s \in [t-n, \,t+n]^c \}  = \hat\E\{y(t)\mid \yb_{[t-n,t)}\vee\yb_{(t ,t+n]}\} 
\end{equation}
%\begin{equation}
%\label{ }
%\hat\E\{\yb_{(t_1,t_2)} \mid \yb(s),\, s\in (t_1,t_2)^c\}=\hat\E\{\yb_{(t_1,t_2)}\mid \yb_{(t_1-n,t_1]} \vee \yb_{[t_2 ,t_2+n)}\}
%\end{equation}
for $t \in[-N+1,N]$. In particular, 
%we should have 
%\begin{equation}
%\label{ }
%\hat\E\{y(t)\mid  \yb(s),\, s \neq t\}  = \hat\E\{y(t)\mid \yb_{[t-n,t)}\vee\yb_{(t ,t+n]}\}
%\end{equation}
%for $t\in[-N+1,N]$, where 
the estimation error
\begin{equation}
 \label{finconjn2}
\begin{split}
  d(t)   & :=  y(t)-\hat\E\{y(t)\mid y(s),\, s \neq t\} \\  
        				 &   = y(t) -\hat\E\{y(t)\mid \yb_{[t-n,t)}\vee\yb_{(t ,t+n]}\}
\end{split}
\end{equation}
must clearly be  orthogonal to all random variables $\{y(s),\, s \neq t\}$; i.e. $\E\{d(t)\overline{y(s)}\}=\sigma^2\delta_{st}$, where $\sigma^2$ is the variance of $d(t)$. 
Then $e(t):=d(t)/\sigma^2$ is the (normalized) conjugate process of $y$ satisfying \eqref{eystar}, i.e., 
\begin{equation}
\label{conjugate}
\E\{e(t)\overline{y(s)}\}=\delta_{ts}.
\end{equation} 
 Since $e(t+k)$ is a linear combination of the components of the random vector $\yb_{[t+k-n,t+k+n]}$, it follows from \eqref{conjugate} that both $e(t+k)$ and  $e(t-k)$ are orthogonal to $e(t)$ for $k >n$. Hence the process $\{e(t)\}$  has   correlation   bandwidth~$n$, i.e.,
\begin{equation}\label{MARecipn}
\E\{e(t+k)\,e(t)^*\} = 0 \quad \text{for $n < |k|  < 2N-n, \; k\in [-N+1,N]$}, 
\end{equation}
and consequently $(\yb,\eb)$ satisfies \eqref{Ay=e}, where $\Gb$ is banded of order $n$, which corresponds to an AR representation \eqref{ME-AR}.

Consequently, the AR solutions of the rational circulant covariance extension problem are precisely the ones corresponding to a reciprocal process $\{y(t)\}$ of order $n$. Next we demonstrate how this representation is connected to the {\em covariance selection problem\/} of Dempster \cite{Dempster} by deriving a generalization of this seminal result.  

Let $J:=\{j_1,\dots,j_p\}$ and $K:=\{k_1,\dots,k_q\}$ be two subsets of $\{-N+1,-N+2,\dots, N\}$, and define $\yb_J$ and $\yb_K$ as the subvectors of $\yb=(y_{-N+1},y_{-N+2}, \cdots, y_N)\Tr$ with indices in $J$ and $K$, respectively. Moreover, let 
\begin{displaymath}
\check{\Yb}_{J,K}:=  \text{span} \{ y(t); \; t\notin J,\; t\notin K\}= \check{\Yb}_{J}\cap\check{\Yb}_{K},
\end{displaymath}
where $\check{\Yb}_{J}:=\text{span} \{ y(t); \; t\notin J\}$. With a slight misuse of notation, we shall write 
\begin{equation}\label{CovSel}
\yb_J \perp \yb_K \mid  \check{\Yb}_{J,K},
\end{equation}
to mean that the subspaces spanned by the components of $\yb_J$ and $\yb_K$, respectively, are conditionally orthogonal given $\check{\Yb}_{J,K}$. This condition can be characterized in terms of the inverse of the covariance matrix $\Sigmab:=\E\{\yb\yb^*\}=\bmat  \sigma_{ij} \emat_{i,j=-N+1}^N$ of $y$.

\begin{thm}\label{genDempsterthm}
Let $\Gb:=\Sigmab^{-1}=\bmat  g_{ij} \emat_{i,j=1}^N$ be  the   concentration matrix  of the random vector $y$. Then the conditional orthogonality relation \eqref{CovSel} holds if and only if\/ $g_{jk}=0$  for all $(j,k)\in J\times K$.
\end{thm}

\begin{proof}
Let $E_J$ be the $2N\times 2N$ diagonal matrix with ones in the positions $(j_1,j_1),\dots,\\(j_m,j_m)$ and zeros elsewhere and let $E_K$ be defined similarly in terms of index set $K$. Then $\check{\Yb}_{J}$ is spanned by the components of $\yb-E_J\yb$ and $\check{\Yb}_{K}$ by the components of $\yb-E_K\yb$. Let 
\begin{displaymath}
\tilde{\yb}_K := \yb_K- \hat\E\{\yb_K\mid\check{\Yb}_{K}\},
\end{displaymath}
and note that its $q\times q$ covariance matrix 
\begin{equation*}
\tilde{\Sigmab}_K := \E\{\tilde{\yb}_K \tilde{\yb}_K^*\}
\end{equation*}
must be positive definite, for otherwise some linear combination of the components of $\yb_K$ would belong yo $\check{\Yb}_{K}$. Let $\tilde{\yb}_K=G_K\yb$ for some $q\times 2N$ matrix $G_K$. Since $\tilde{\yb}_K\perp\check{\Yb}_{K}$, 
\begin{equation*}
\E\{\tilde{\yb}_K(\yb-E_K\yb)^*\}=0
\end{equation*}
 and therefore $\E\{\tilde{\yb}_K\yb^*\}= G_K\Sigmab$ must be equal to $\E\{\tilde{\yb}_K(E_K\yb)^*\}$, which by $\tilde{\yb}_K \in \check{\Yb}_{K}^{\perp}$, in turn equals 
\begin{displaymath}
\E\{\tilde{\yb}_K(E_K\yb)^*\}=\E\{\tilde{\yb}_K\hat\E\{(E_K\yb)^*\mid \check{\Yb}_{K}^{\perp}\}\}.
\end{displaymath}
However, since the  nonzero components of $\hat\E\{E_K\yb\mid \check{\Yb}_{K}^{\perp}\}$ are those of $\tilde\yb_K$, there is an $2N\times q$ matrix $\Pi_K$ with  the unit vectors $e^{\prime}_{k_i}$,  $i=1,\ldots,q$, as the rows such that
\begin{displaymath}
\hat\E\{E_K\yb\mid \check{\Yb}_{K}^{\perp}\}=\Pi_K\tilde{\yb}_K,
\end{displaymath}
and hence 
\begin{displaymath}
\E\{\tilde{\yb}_K(E_K\yb)^*\}=\E\{\tilde{\yb}_K\tilde{\yb}_K^*\}\Pi_K^*=\tilde{\Sigmab}_K\Pi_K^*. 
\end{displaymath}
Consequently, $G_K\Sigmab=\tilde{\Sigmab}_K\Pi_K^*$, i.e., 
\begin{displaymath}
G_K=\tilde{\Sigmab}_K\Pi_K^*\Sigmab^{-1}.
\end{displaymath}
In the same way, $\tilde{\yb}_J=G_J\yb$, where $G_J$ is the $q\times 2N$ matrix 
\begin{displaymath}
G_J=\tilde{\Sigmab}_J\Pi_J^*\Sigmab^{-1},
\end{displaymath}
and therefore 
\begin{displaymath}
\E\{\tilde{\yb}_J\tilde{\yb}_K^*\} =\tilde{\Sigmab}_J\Pi_J^*\Sigmab^{-1}\Pi_K\tilde{\Sigmab}_K,
\end{displaymath}
which is zero if and only if $\Pi_J^*\Sigmab^{-1}\Pi_K=0$, i.e., $g_{jk}=0$  for all $(j,k)\in J\times K$.

It remains to show that $\E\{\tilde{\yb}_J\tilde{\yb}_K^*\}=0$ is equivalent to \eqref{CovSel}, which in view of  \eqref{condorthPerp2}, can be written
\begin{equation*}
\E\left\{\hat\E\{\yb_J\mid \check{\Yb}_{J,K}\}\{\hat\E\{\yb_K\mid \check{\Yb}_{J,K}\}^*\right\}=\{\yb_J\yb_K^*\}.
\end{equation*}
However,
\begin{displaymath}
E\{\tilde{\yb}_J\tilde{\yb}_K^*\}=E\{\yb_J\yb_K^*\} -\E\left\{\hat\E\{\yb_J\mid \check{\Yb}_J\}\{\hat\E\{\yb_K\mid \check{\Yb}_K\}^*\right\},
\end{displaymath}
so the proof will complete if we show that 
\begin{equation}
\label{Lemma2.6.9}
\E\left\{\hat\E\{\yb_J\mid \check{\Yb}_J\}\{\hat\E\{\yb_K\mid \check{\Yb}_K\}^*\right\}=
\E\left\{\hat\E\{\yb_J\mid \check{\Yb}_{J,K}\}\{\hat\E\{\yb_K\mid \check{\Yb}_{J,K}\}^*\right\}
\end{equation}
the proof of which follows precisely the lines of Lemma 2.6.9 in \cite[p. 56]{LPbook}. 
\end{proof}

Taking $J$ and $K$ to be singletons we recover as a special case Dempster's original result \cite{Dempster}. 

To connect back to Definition~\ref{def:Recn} of a reciprocal process of order $n$, use the  equivalent condition \eqref {RecMod} so that, with $J=\{t\}$ and $K= [t-n, \,t+n]^c $,    $\yb_J=y(t)$ and $\yb_{K}$ are conditionally orthogonal given  $\check{\Yb}_{J,K}= \yb_{[t-n,t)}\vee\yb_{(t ,t+n]}$.
%we take $J=(t_1,t_2)$ and $K=(t_1,t_2)^c$ so that $\yb_J=\yb(t_1,t_2)$, $\yb_K=\yb(t_1,t_2)^c$ and $\check{\Yb}_{J,K}= \yb_{[t-n,t)}\vee\yb_{(t ,t+n]}$. 
Then $J\times K$ is the set $\big \{ t \times  [t-n, \,t+n]^c \,;\, t\in (-N,\,N]\,\big \}$, and hence
Theorem~\ref{genDempsterthm} states precisely that the circulant matrix $\Gb$ is banded of order $n$. We stress that in general $\Gb=\Sigmab^{-1}$ is not banded, as the underlying process $\{y(t)\}$ is not reciprocal of degree $n$,
and we then have an ARMA representation as explained in Section~\ref{bilateralsec}.

\section{Detemining $\Pb$ with the help of logarithmical moments}\label{logarithmicsec}

We have shown that the solutions of the circulant rational covariance extension problem, as well as the corresponding bilateral ARMA models, are completely parameterized by $P\in\mathfrak{P}_+(N)$, or, equivalently, by their corresponding banded circulant matrices $\Pb$. This leads to the question of how to determine the $\Pb$ from given data.

To this end, suppose that we are also given the logarithmic moments
\begin{equation}
\label{cepstrum}
\gamma_k=\int_{-\pi}^\pi e^{ik\theta}\log\Phi(e^{i\theta})d\nu, \quad k=1,2,\dots,n.
\end{equation}
In the setting of the classical trigonometric moment problem such moments are known as {\em cepstral coefficients}, and in speech processing, for example, they are estimated from observed data for purposes of  design.

Following \cite{LPcirculant} and, in the context of the trigonometric moment problem, \cite{Musicus,BEL1,BEL2,PE}, we normalize the elements in $\mathfrak{P}_+(N)$ to define $\tilde{\mathfrak{P}}_+(N):=\{P\in\mathfrak{P}_+(N)\mid p_0=1\}$ and consider the problem to find a nonnegative integrable $\Phi$ maximizing
\begin{equation}
\label{I(Phi)}
\mathbb{I}(\Phi) =\int_{-\pi}^\pi \log\Phi(e^{i\theta})d\nu =\frac{1}{2N}\sum_{j=-N+1}^N \log\Phi(\zeta_j)
\end{equation}
subject to the moment constraints \eqref{discretemoments} and \eqref{cepstrum}. It is shown in \cite{LPcirculant} that  
if there is a maximal $\Phi$ that is positive on the unit circle, it is given by 
\begin{equation}
\label{Phi2}
\Phi(\zeta)=\frac{P(\zeta)}{Q(\zeta)},
\end{equation}
where $(P,Q)$ is the unique solution of the dual problem to minimize
\begin{equation}
\label{J(P,Q)}
\mathbb{J}(P,Q)=\langle \cb,\qb\rangle -\langle \gammab,\pb\rangle + \int_{-\pi}^\pi P(e^{i\theta})\log\left(\frac{P(e^{i\theta})}{Q(e^{i\theta})}\right)d\nu
\end{equation}
over all $(P,Q)\in\tilde{\mathfrak{P}}_+(N)\times\mathfrak{P}_+(N)$, where $\gammab=(\gamma_0,\gamma_1,\dots,\gamma_n)$ and $\pb=(p_0, p_1,\dots,p_n)$ with $\gamma_0=0$ and $p_0=1$.

The problem is that the dual problem might have a minimizer on the boundary so that there is no stationery point in the interior, and then the constraints \eqref{cepstrum} will in general not be satisfied  \cite{LPcirculant}. Therefore the problem needs to be regularized in the style of \cite{PEthesis}. More precisely, we consider the regularized problem to minimize
\begin{equation}
\label{J(P,Q)reg}
\mathbb{J}_\lambda(P,Q) =\mathbb{J}(P,Q) -\lambda\int_{-\pi}^\pi \log P(e^{i\theta})d\nu
\end{equation}
for some suitable $\lambda>0$  over all $(P,Q)\in\tilde{\mathfrak{P}}_+(N)\times\mathfrak{P}_+(N)$. Setting $\mathbf{J}_\lambda(\Pb,\Qb):=2N\mathbb{J}_\lambda(P,Q)$, \eqref{J(P,Q)reg} can be written
\begin{equation}
\label{Jlambda}
\mathbf{J}_\lambda(\Pb,\Qb) =\text{\rm tr}\{\Cb\Qb\} -\text{\rm tr}\{\Gammab\Pb\} + \text{\rm tr}\{\Pb\log\Pb\Qb^{-1}\} - \lambda\,\text{\rm tr}\{\log\Pb\},
\end{equation}
where $\Gammab$ is the Hermitian circulant matrix with symbol
\begin{equation}
\label{M(z)}
\Gamma(\zeta)=\sum_{k=-n}^n \gamma_k \zeta^{-k}, \quad \gamma_{-k}=\bar{\gamma}_k.
\end{equation}

Therefore, in the circulant matrix form, the regularized dual problem amounts to minimizing  \eqref{Jlambda} over all banded Hermitian circulant matrices $\Pb$ and $\Qb$ of order $n$ subject to $p_0=1$. It is shown in \cite{LPcirculant}  that 
\begin{equation}
\label{Sigmab2}
\Sigmab=\Qb^{-1}\Pb,
\end{equation}
or, equivalently in symbol form \eqref{Phi2}, maximizes 
\begin{equation}
\label{ }
\mathbf{I}(\Sigmab)=\text{\rm tr}\{\log\Sigmab\} =\log\det\Sigmab,
\end{equation}
or, equivalently \eqref{I(Phi)}, subject to \eqref{discretemoments} and \eqref{cepstrum}, the latter constraint modified so that the logarithmic moment $\gamma_k$ is exchanged for $\gamma_k+\varepsilon_k$, $k=1,2,\dots,n$, where
\begin{equation}
\label{epsilon}
\varepsilon_k=\int_{-\pi}^\pi e^{ik\theta}\frac{\lambda}{\hat P(e^{i\theta})}d\nu= \frac{\lambda}{2N}\text{\rm tr}\{\Sb^k\hat\Pb^{-1}\}, 
\end{equation}
$\hat P$ being the optimal $P$. 

The following example from \cite{LPcirculant}, provided by Chiara Masiero, illustrates the advantages of this procedure. We start from an ARMA model with $n=8$ poles and three zeros distributed as in Figure \ref{figure3}, from which we compute $\cb=(c_0,c_1,\dots,c_n)$ and  $\gammab=(\gamma_1,\dots,\gamma_n)$ for various choices of the order $n$. 
\begin{figure}
\centering
\includegraphics[height=5.4cm,width=5.4cm]{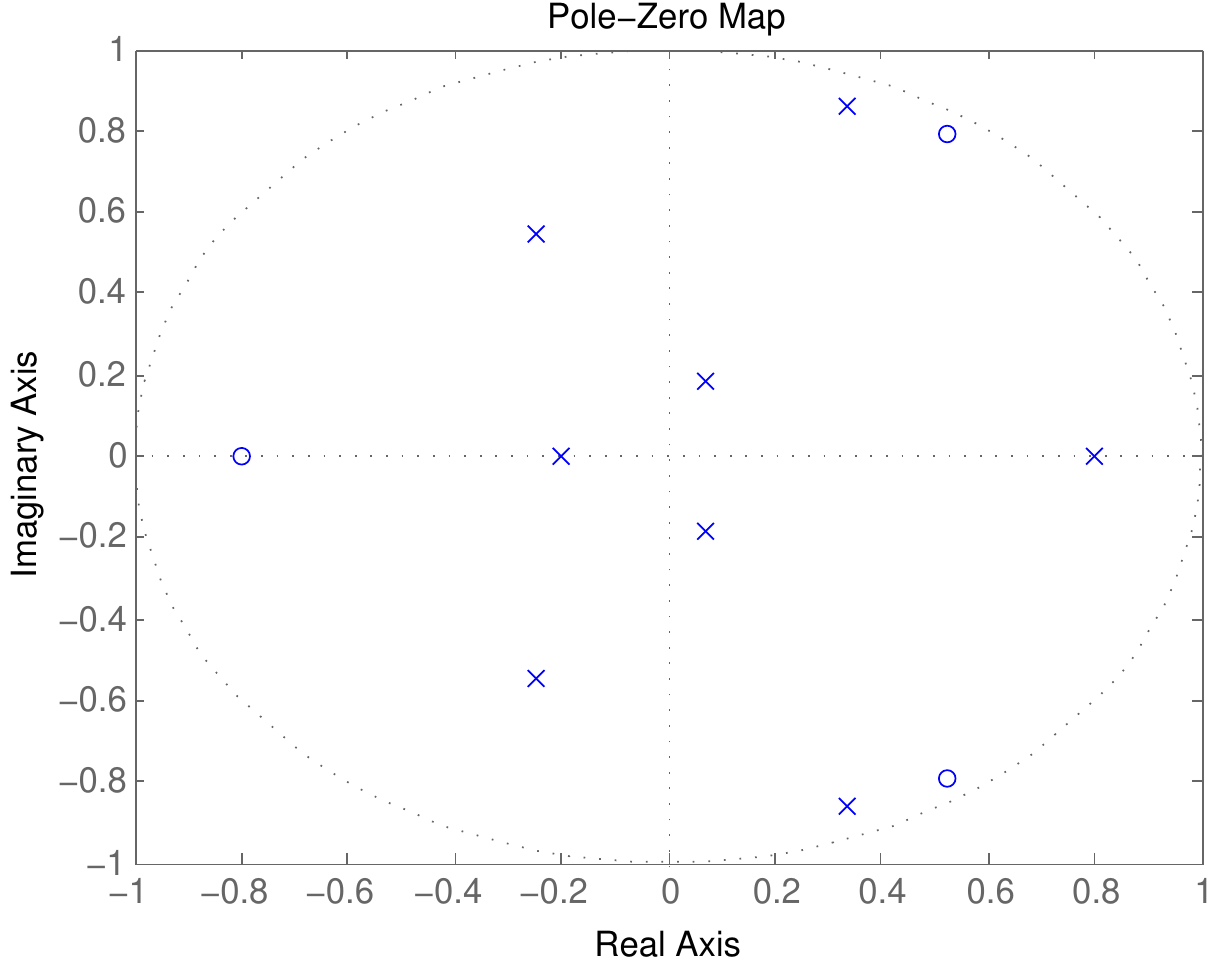}
\caption{\small Poles and zeros of true ARMA model.}
\label{figure3}
\end{figure}
First we determine the maximum entropy solution from $\cb$ with $n=12$ and $N=1024$. The resulting spectral function $\Phi$ is depicted in the left plot of Figure~\ref{figures5} together with the true spectrum. Next we compute $\Phi$ by the procedure in this section using $\cb$ and $\gammab$ with $n=8$ and $N=128$. The result is depicted in the right plot of Figure~\ref{figures5} again together with the true spectrum. This illustrates the advantage of bilateral ARMA modeling as compared to bilateral AR modeling, as a much lower value on $N$ provides a better approximation, although $n$ is smaller. 

\begin{figure}
\centering
\includegraphics[height=4cm]{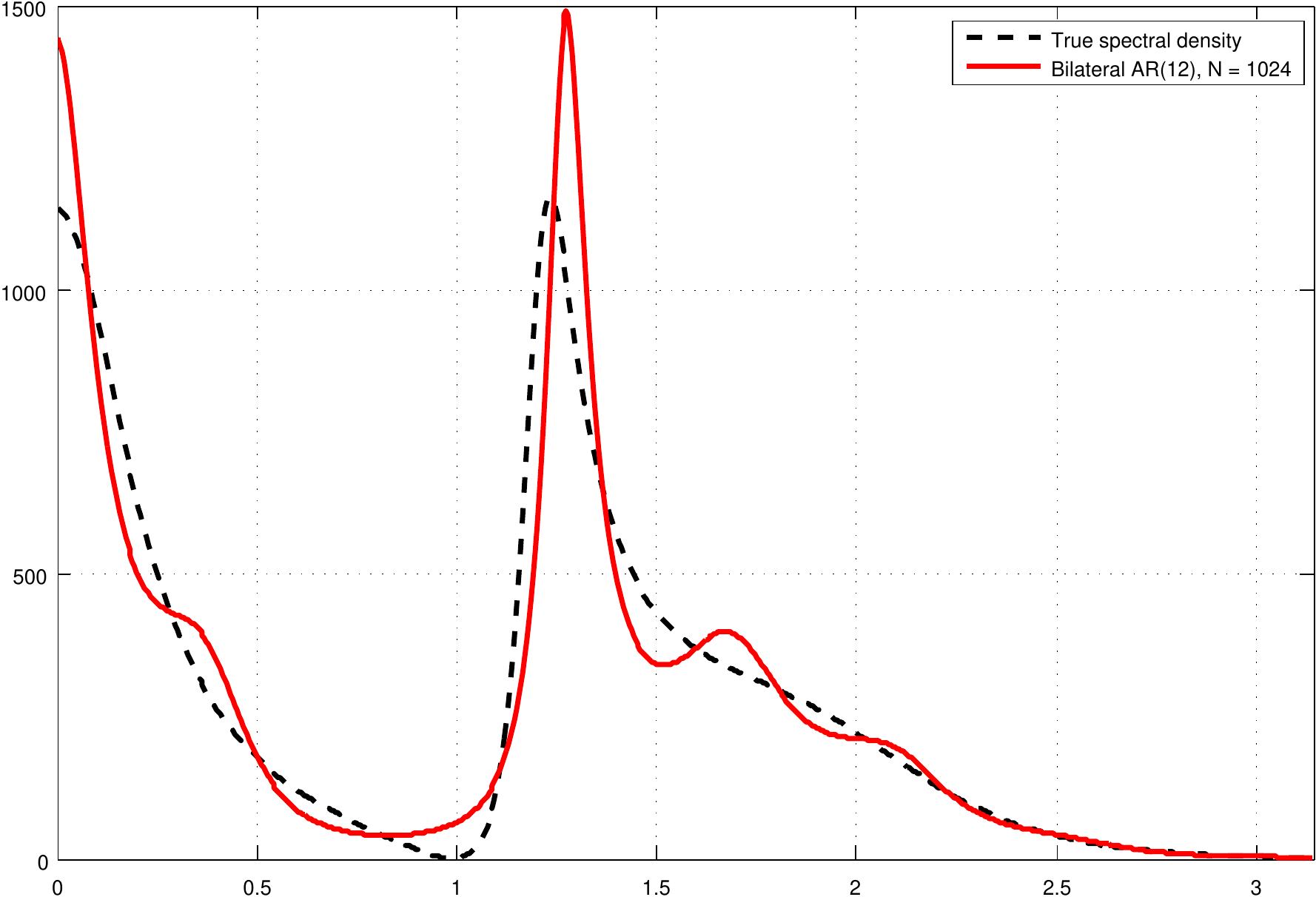}
\qquad
\includegraphics[height=4cm]{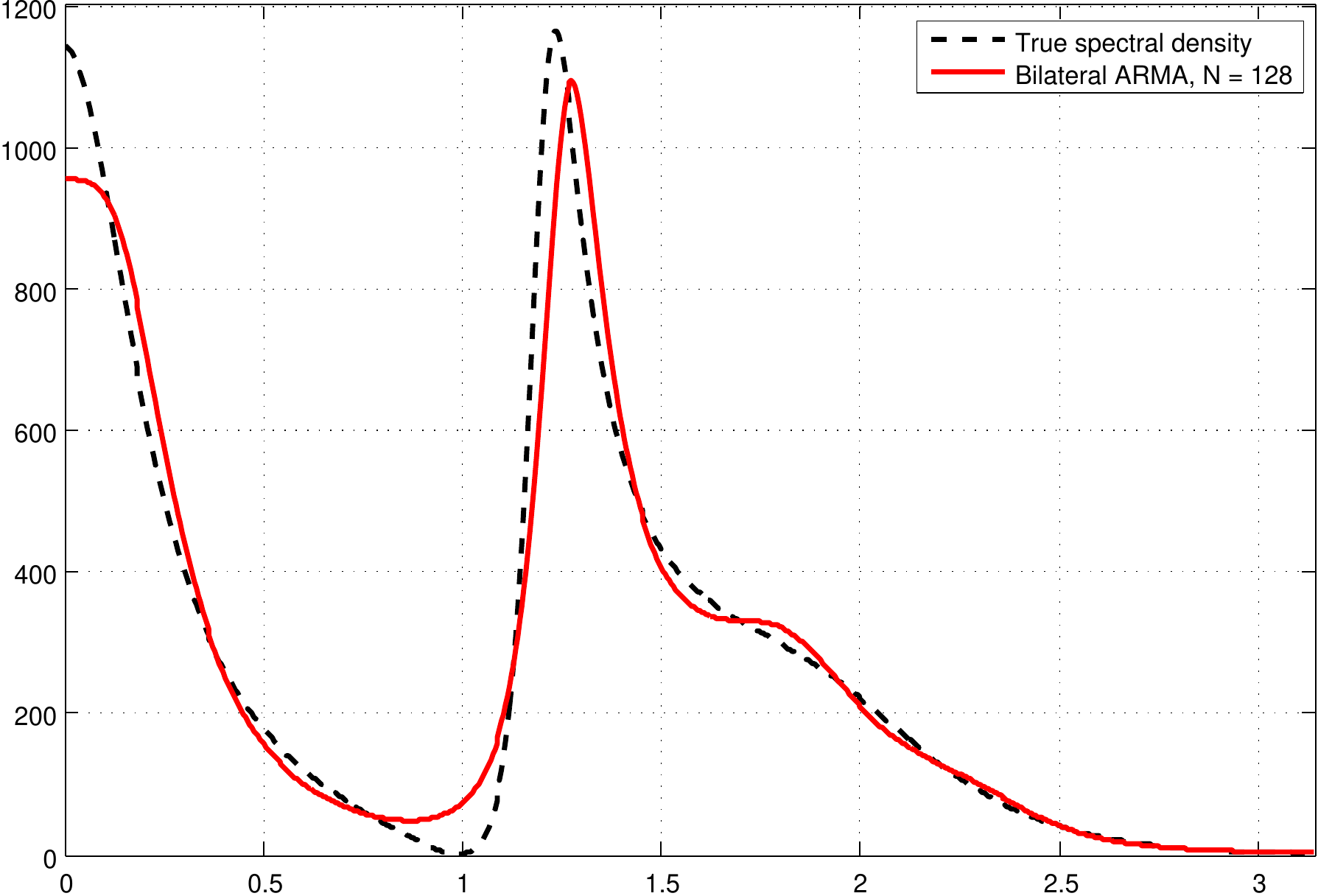}
\caption{\small  Bilateral approximations with true spectrum (dashed): (left) bilateral AR with $n=12$  and $N=1024$; (right) bilateral ARMA with $n=8$ and $N=128$ using both covariance and logarithmic moment estimates.}
\label{figures5}
\end{figure}

\section{Extensions to multivariate case}\label{multivariatesec}

To simplify notation we have so far restricted our attention to scalar stationary periodic processes. We shall now demonstrate that most of the results can be simply extended to the multivariate case, provided we restrict the analysis to scalar pseudo-polynomials $P(\zeta)$. In fact, most of the equations in the previous section will remain intact if we allow ourselves to interpret the scalar quantities as matrix-valued ones. 

Let $\{y(t)\}$ be a zero-mean stationary $m$-dimensional process $\{y(t)\}$  defined on $\Zbb_{2N}$;  i.e., a stationary process defined on a finite interval $[-N+1,\,N]$ of the integer line $\Zbb$ and extended to all of $\Zbb$ as a periodic stationary process with  period $2N$. Moreover, let $C_{-N+1},C_{-N+2},\dots,C_{N}$ be the $m\times m$ covariance lags $C_k:=\E\{ y(t+k)y(t)^*\}$, and define its discrete Fourier transformation 
\begin{equation} \label{c2Phi}
\Phi(\zeta_{j}) := \sum_{k=-N+1 }^{N }\, C_k \zeta_{j}^{-k} \,,\qquad j=-N+1,\dots , N,
\end{equation}
which is a positive, Hermitian matrix-valued function of $\zeta$. Then, by the inverse discrete Fourier transformation,
\begin{equation} \label{Phi2c}
C_k = \frac{1}{2N}\sum_{j=-N+1 }^{N  }\zeta_{j}^k \Phi(\zeta_j)
=\int_{-\pi}^\pi e^{ik\theta}\Phi(e^{i\theta})  d\nu,\quad k =  -N+1, \dots ,N,
\end{equation}
where the Stieljes measure $d\nu$ is given by \eqref{nu}. The $m\times m$ matrix  function $\Phi$ is the {\em spectral density\/} of the vector process $y$. In fact, let
\begin{equation}
\label{yDFT}
\hat{y}(\zeta_k):= \sum_{t=-N+1}^N y(t)\zeta_k^{-t}, \quad k=-N+1,\dots, N,
\end{equation}
be the discrete Fourier transformation of the process $y$. Since 
\begin{displaymath}
\frac{1}{2N}\sum_{t=-N+1}^N (\zeta_k\zeta_\ell^*)^t =\delta_{k\ell}
\end{displaymath}
by \eqref{delta}, the random variables \eqref{yDFT} are uncorrelated, and 
\begin{equation}
\label{yhatyhat}
\frac{1}{2N}\E\{ \hat{y}(\zeta_k)\hat{y}(\zeta_\ell)^*\}=\Phi(\zeta_{k})\delta_{k\ell}.
\end{equation}
This yields a spectral representation of $y$ analogous to the usual one, namely
\begin{equation}
\label{ }
y(t)=\frac{1}{2N}\sum_{k=-N+1}^N \zeta_k^t\,\hat{y}(\zeta_k)=\int_{-\pi}^\pi e^{ik\theta}d\hat{y}(\theta),
\end{equation}
where $d\hat{y}:=\hat{y}(e^{i\theta})d\nu$.

Next, we define the class $\mathfrak{P}_+^{(m,n)}(N)$ of $m\times m$ Hermitian pseudo-polynomials
\begin{equation}
\label{Qmatrix}
Q(\zeta)=\sum_{k=-n}^n Q_k\zeta^{-k}, \quad Q_{-k}=Q_k^*
\end{equation}
of degree at most $n$ that are positive definite on the discrete unit circle $\mathbb{T}_{2N}$, and let  $\mathfrak{P}_+^{(m,n)}\subset\mathfrak{P}_+^{(m,n)}(N)$ be the subset of all \eqref{Qmatrix} such that $Q(e^{i\theta})$ is positive define for all $\theta\in [-\pi,\pi]$. Moreover let $\mathfrak{C}_+^{(m,n)}(N)$ be the dual cone of all $C=(C_0,C_1,\dots,C_n)$ such that
\begin{displaymath}
\langle C,Q\rangle :=\sum_{k=-n}^n \trace\{C_k Q_k^*\} >0 \quad \text{for all $Q\in\overline{\mathfrak{P}_+^{(m,n)}(N)}\setminus\{0\}$},
\end{displaymath}
and let $\mathfrak{C}_+^{(m,n)}\supset\mathfrak{C}_+^{(m,n)}(N)$ be defined as the dual cone of $\mathfrak{P}_+^{(m,n)}$. Analogously to the scalar case it can be shown that $C\in\mathfrak{C}_+^{(m,n)}$ if and only if the block-Toeplitz matrix
\begin{equation}
\label{blockToeplitz}
\Tb_n=\begin{bmatrix} C_0&C_1^*&C_2^*&\cdots&C_n^*\\
				C_1&C_0&C_1^*&\cdots& C_{n-1}^*\\
				C_2&C_1&C_0&\cdots&C_{n-2}^*\\
				\vdots&\vdots&\vdots&\ddots&\vdots\\
				C_n&C_{n-1}&C_{n-2}&\cdots&C_0
 		\end{bmatrix}
\end{equation}
is positive definite \cite{LMPcirculantMult}, a condition that is necessary, but in general not sufficient, for $C\in\mathfrak{C}_+^{(m,n)}(N)$ to hold.

The basic problem is now the following. Given the sequence $C=(C_0,C_1,\dots,C_n)\in\mathfrak{C}_+^{(m,n)}(N)$ of $m\times m$ covariance lags, find an extension $C_{n+1},C_{n+2},\dots,C_N$ with $C_{-k}=C_k^*$ such that the spectral function  $\Phi$ defined by \eqref{c2Phi} has the rational form
\begin{equation}
\label{Phimatrix=PQinv}
\Phi(\zeta)=P(\zeta)Q(\zeta)^{-1}, \quad P\in\mathfrak{P}_+^{(1,n)}(N), \, Q\in\mathfrak{P}_+^{(m,n)}(N).
\end{equation} 

\begin{thm}\label{mainthm(matrix)}
Let $C\in\mathfrak{C}_+^{(m,n)}(N)$. Then, for each $P\in\mathfrak{P}_+^{(1,n)}(N)$, there is a unique $Q\in\mathfrak{P}_+^{(m,n)}(N)$ such that 
\begin{equation}
\label{Phi=P/Qmatrix}
\Phi=PQ^{-1}
\end{equation}
satisfies the moment conditions 
\begin{equation}
\label{matrixmoments}
\int_{-\pi}^\pi e^{ik\theta}\Phi(e^{i\theta})d\nu =C_k, \quad k=0,1,\dots,n.
\end{equation}
\end{thm}

Theorem~\ref{mainthm(matrix)} is a direct consequence of the following theorem, which also provides an algorithm for computing the solution. 

\begin{thm}\label{optthm(matrix)}
For each $(C,P)\in\mathfrak{C}_+^{(m,n)}(N)\times\mathfrak{P}_+^{(1,n)}(N)$, the problem to maximize the functional
\begin{equation}
\label{matrixprimal}
\mathbb{I}_P(\Phi) =\int_{-\pi}^\pi  P(e^{i\theta})\log\det \Phi(e^{i\theta})d\nu
\end{equation}
subject to the moment conditions \eqref{matrixmoments} has a unique solution $\hat{\Phi}$, and it has the form   
\begin{equation}
\label{Phiopt}
\hat{\Phi}(z)=P(z)\hat{Q}(z)^{-1},
\end{equation}  
where $\hat{Q}\in\mathfrak{P}_+^{(m,n)}(N)$ is the unique solution to the dual problem to minimize 
\begin{equation}\label{dual}
\mathbb{J}_P(Q)= \langle C,Q\rangle -\int_{-\pi}^\pi  P(e^{i\theta})\log \det Q(e^{i\theta})d\nu
\end{equation}
over all $Q\in\mathfrak{P}_+^{(m,n)}(N)$.
\end{thm}

The proofs of Theorems~\ref{mainthm(matrix)} and \ref{optthm(matrix)} follow the lines of \cite{LPcirculant}. 
%and will be given in the appendix. 
It can also be shown that the moment map sending $Q\in\mathfrak{P}_+^{(m,n)}(N)$ to $C\in\mathfrak{C}_+^{(m,n)}(N)$ is a diffeomorphism. 

To formulate a matrix version of Theorems~\ref{mainthm(matrix)} and \ref{optthm(matrix)} we need to introduce (Hermitian) block-circulant matrices
\begin{equation}
\label{S2C}
\Mb =\sum_{k=-N+1}^N S^{-k}\otimes M_k, \quad M_{-k}=M_k^*
\end{equation}
where $\otimes$ is the Kronecker product and $S$ is the nonsingular $2N\times 2N$  cyclic shift matrix \eqref{Sb}. The notation $\Sb$ will now be reserved for the $2mN\times 2mN$ block-shift matrix
\begin{equation}
\label{Sblarge}
\Sb = S\otimes I_m =\left[\begin{array}{cccccc}0 & I_m & 0 & \dots & 0 \\0 & 0 & I_m &  \dots & 0  \\\vdots & \vdots & \vdots & \ddots & \vdots \\0 & 0 & 0 &  0 & I_m \\I_m & 0 & 0 & 0  & 0\end{array}\right].
\end{equation}
As before $\Sb^{2N}=\Sb^0=\Ib :=I_{2mN}$, $\Sb^{k+2N}=\Sb^k$, and  $\Sb^{2N-k}=\Sb^{-k}=(\Sb^k)\Tr$. Moreover
\begin{equation}
 \label{circulantcondition}
\Sb\Mb \Sb^*=\Mb
\end{equation}
is both necessary and sufficient for $\Mb$ to be $m\times m$ block-circulant. The symbol of $\Mb$ is the $m\times m$ pseudo-polynomial
\begin{equation}
\label{symbol}
M(\zeta)=\sum_{k=-N+1}^N M_k \zeta^{-k}, \quad M_{-k}=M_k^*. 
\end{equation}
We shall continue using the notation 
\begin{equation}
\label{M_C}
\Mb:=\Circ\{ M_0,M_1,M_2,\dots, M_N,M_{N-1}^*,\dots,M_1^*\}
\end{equation}
also for (Hermitain) block-circulant matrices. 

The problem can now be reformulated in the following way. Given the banded block-circulant matrix 
\begin{equation}
\label{Cbmatrix}
\Cb =\sum_{k=-n}^n S^{-k}\otimes C_k, \quad C_{-k}=C_k^*
\end{equation}
of order $n$, find an  extension $C_{n+1},C_{n+2},\dots,C_N$ such that the block-circulant matrix
\begin{equation}
\label{Sigmabmatrix}
\Sigmab =\sum_{k=-N+1}^N S^{-k}\otimes C_k, \quad C_{-k}=C_k^*
\end{equation}
has the symbol \eqref{Phimatrix=PQinv}. 

To proceed we need a block-circulant version of Lemma~\ref{diagonalizationlem}. 

\begin{lem}\label{matrixdiagonalizationlem}
Let $\Mb$ be a block-circulant matrix with symbol $M(\zeta)$. Then 
\begin{equation}
\label{matrixMdiag}
\Mb=\Fb^*\text{\rm diag}\big(M(\zeta_{-N+1}),M(\zeta_{-N+2}),\dots,M(\zeta_N)\big)\Fb,
\end{equation}
where $\Fb$ is the unitary $2mN\times 2mN$ matrix
\begin{equation}
\label{Fmatrix}
\Fb= \frac{1}{\sqrt{2N}}\left[\begin{array}{cccc}\zeta_{-N+1}^{N-1}I_m & \zeta_{-N+1}^{N-2}I_m & \cdots & \zeta_{-N+1}^{-N} I_m\\ \vdots & \vdots & \cdots & \vdots \\\zeta_{0}^{N-1}I_m & \zeta_{0}^{N-2}I_m & \cdots & \zeta_{0}^{-N} I_m \\ \vdots & \vdots & \cdots & \vdots \\\zeta_{N}^{N-1}I_m & \zeta_{N}^{N-2} I_m& \cdots & \zeta_{N}^{-N}I_m \end{array}\right] .
\end{equation}
Moreover, if $M(\zeta_k)$ is positive definite for all $k$, then
\begin{equation}
\label{matrixlogM}
\log\Mb=\Fb^*\text{\rm diag}\big(\log M(\zeta_{-N+1}),\log M(\zeta_{-N+2}),\dots,\log M(\zeta_N)\big)\Fb,
\end{equation}
where $\text{\rm diag}$ stands for block diagonal.
\end{lem}

The proof of Lemma~\ref{matrixdiagonalizationlem} will be omitted, as it follows the same lines as that of Lemma~\ref{diagonalizationlem} with straight-forward modification to the multivariate case. 
Clearly the inverse
\begin{equation}
\label{matrixMdiaginv}
\Mb^{-1}=\Fb^*\text{\rm diag}\big(M(\zeta_{-N+1})^{-1},M(\zeta_{-N+2})^{-1},\dots,M(\zeta_N)^{-1}\big)\Fb
\end{equation}
is also block-circulant, and 
\begin{equation}
\label{matrixSbrepr}
\Sb =\Fb^*\text{\rm diag}\big(\zeta_{-N+1}I_m,\zeta_{-N+2}I_m,\dots,\zeta_N I_m\big)\Fb.
\end{equation}
However, unlike the scalar case block-circulant matrices do not commute in general. 

Given Lemma~\ref{matrixdiagonalizationlem}, we are now in a position to reformulate Theorems~\ref{mainthm(matrix)} and \ref{optthm(matrix)} in matrix from. 

\begin{thm}\label{blockmainthm_matrix}
Let $C\in\mathfrak{C}_+^{(m,n)}(N)$, and let $\Cb$ be the corresponding block-circulant matrix \eqref{Cbmatrix} and 
\eqref{blockToeplitz} the corresponding block-Toeplitz matrix.  Then, for each  positive-definite banded $2mN\times 2mN$ block-circulant matrices 
\begin{equation}
\label{matrixPbmatrix}
\Pb =\sum_{k=-n}^n S^{-k}\otimes p_k I_m, \quad p_{-k}=\bar{p}_k
\end{equation}
of order $n$, where $P(\zeta)=\sum_{k=-n}^n p_k \zeta^{-k}\in\mathfrak{P}_+^{(1,n)}(N)$, there is a unique sequence $Q=(Q_0,Q_1,\dots,Q_n)$ of $m\times m$ matrices defining a positive-definite banded $2mN\times 2mN$ block-circulant matrix
\begin{equation}
\label{matrixQbmatrix}
\Qb =\sum_{k=-n}^n S^{-k}\otimes Q_k, \quad Q_{-k}=Q_k^*
\end{equation}
of order $n$ such that 
\begin{equation}
\label{mSigmab=QbinvPb}
\Sigmab=\Qb^{-1}\Pb
\end{equation}
is a block-circulant extension \eqref{Sigmabmatrix} of $\Cb$. The block-circulant matrix  \eqref{mSigmab=QbinvPb} is the unique maximizer of the function
\begin{equation}
\label{mprimal_matrix}
\mathcal{I}_{\Pb}(\Sigmab) = \trace(\Pb\log\Sigmab)
\end{equation}
subject to 
\begin{equation}
\label{mmomentcondmatrixb}
\Eb_n\Tr\Sigmab \Eb_n =\Tb_n, \quad \text{where\;} \Eb_n =\begin{bmatrix}\Ib_{mn}\\{\bold 0}\end{bmatrix}.
\end{equation}
Moreover, $\Qb$ is the unique optimal solution of the problem to minimize
\begin{equation}
\label{dual_matrix2}
\mathcal{J}_{\Pb}(\Qb)= \trace(\Cb\Qb) - \trace(\Pb\log\Qb)
\end{equation}
over all positive-definite banded $2mN\times 2mN$ block-circulant matrices \eqref{matrixQbmatrix} of order $n$.  The functional $\mathcal{J}_{\Pb}$ is strictly convex.  
\end{thm} 

For $\Pb=\Ib$ we obtain the maximum-entropy solution considered in \cite{Carli-FPP}, where the primal problem to maximize $\mathcal{I}_{\Ib}$ subject to \eqref{mmomentcondmatrixb} was presented. In \cite{Carli-FPP} there was also an extra constraint \eqref{circulantcondition}, which, as we can see, is not needed, since it is automatically fulfilled. For this reason the dual problem presented in  \cite{Carli-FPP} is more complicated than merely minimizing $\mathcal{J}_{\Ib}$. 

Next suppose we are also given the (scalar) logarithmic moments \eqref{cepstrum} and that  $C\in\mathfrak{C}_+^{(m,n)}(N)$. Then, if the problem to maximize $\text{\rm tr}\{\log\Sigmab\}$ subject to \eqref{mmomentcondmatrixb} and \eqref{cepstrum}  over all positive-definite block-circulant matrices \eqref{Sigmabmatrix} has a solution, then it has the form 
\begin{equation}
\label{PbQb12Sigmab}
\Sigmab=\Qb^{-1}\Pb
\end{equation}
where the $(\Pb,\Qb)$ is a solution of the dual problem to minimize 
\begin{equation}
\label{Jmatrix}
\mathbf{J}(\Pb,\Qb) =\text{\rm tr}\{\Cb\Qb\} -\text{\rm tr}\{\Gammab\Pb\} + \text{\rm tr}\{\Pb\log\Pb\Qb^{-1}\},
\end{equation}
over all positive-definite block-circulant matrices of the type \eqref{matrixPbmatrix} and \eqref{matrixQbmatrix} with the extra constrain $p_0=1$, where $\Gammab$ is the Hermitian circulant matrix with symbol
\begin{equation}
\label{M(z)}
\Gamma(\zeta)=\sum_{k=-n}^n \gamma_k \zeta^{-k}, \quad \gamma_{-k}=\bar{\gamma}_k.
\end{equation}

However, the minimum of \eqref{Jmatrix} may end up on the boundary, in which case the constraint \eqref{cepstrum} may fail to be satisfied. Therefore, as in the scalar case, we need to regularize the problem by instead minimizing 
\begin{equation}
\label{magtrixJlambda}
\mathbf{J}_\lambda(\Pb,\Qb) =\text{\rm tr}\{\Cb\Qb\} -\text{\rm tr}\{\Gammab\Pb\} + \text{\rm tr}\{\Pb\log\Pb\Qb^{-1}\} - \lambda\,\text{\rm tr}\{\log\Pb\}.
\end{equation}
This problem has a unique optimal solution \eqref{PbQb12Sigmab} satisfying \eqref{mmomentcondmatrixb}, but not   \eqref{cepstrum}. The appropriate logarithmic moment constraint is obtained as in the scalar case by exchanging  $\gamma_k$ for $\gamma_k+\varepsilon_k$ for each $k=1,2,\dots,n$, where $\varepsilon_k$ is given by \eqref{epsilon}. 

Again each solution leads to an ARMA model
\begin{equation}
\label{matrixARMA}
\sum_{k=-n}^n Q_k y(t-k) = \sum_{k=-n}^n p_k e(t-k),
\end{equation}
where $\{e(t)\}$ is the conjugate process of $\{y(t)\}$, $Q_0,Q_1,\dots,Q_n$ are $m\times m$ matrices, whereas  $p_0,p_1,\dots,p_n$ are scalar with $p_0=1$. 

We illustrate this theory with a simple example from \cite{LMPcirculantMult}, where a covariance sequence $C:=(C_0,C_1,\dots C_n)$ and a cepstral sequence $\gammab:=(\gamma_1,\gamma_2,\dots,\gamma_n)$ have been computed from a two-dimensional ARMA process with a  spectral density $\Phi:=PQ^{-1}$, where $P$ is a scalar pseudo-polynomial of degree three and $Q$ is a $2\times 2$ matrix-valued pseudo-polynomial of degree $n=6$.
Its zero and poles are illustrated in Fig.~\ref{fig:ARMA_n6_pzmap}. 
\begin{figure}
	\centering
	\includegraphics[width=0.35\textwidth]{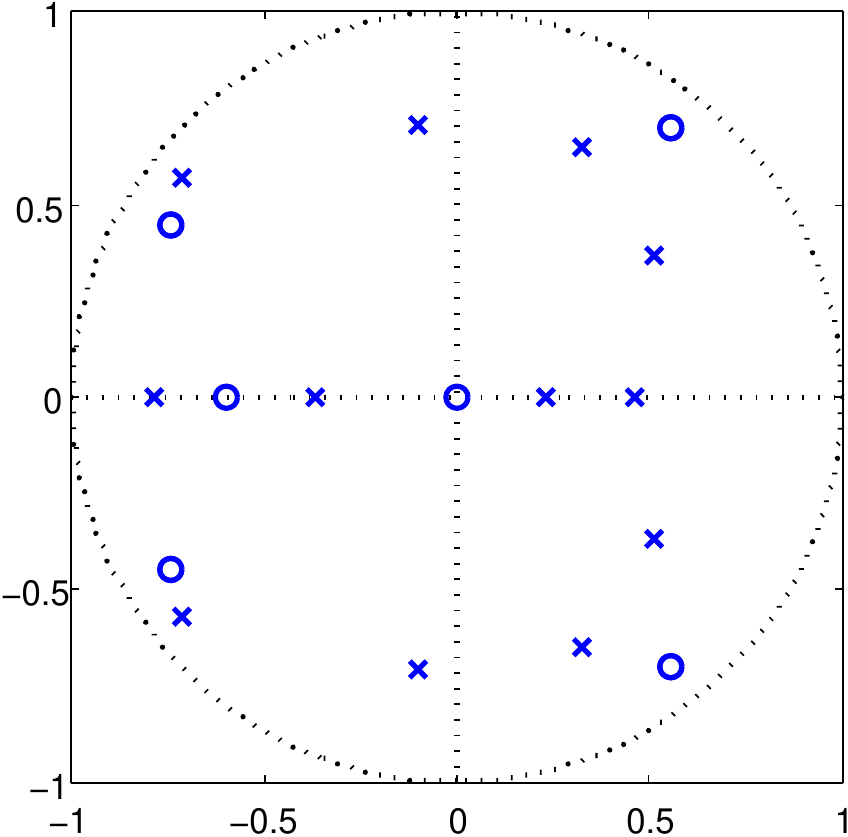} 
	\caption{Poles and zeros of an ARMA $2\times 2$ model of order $n=6$.}
	\label{fig:ARMA_n6_pzmap}
\end{figure}
\begin{figure}[h!]
%\begin{figure}
	\centering
	\includegraphics[width=0.60\textwidth]{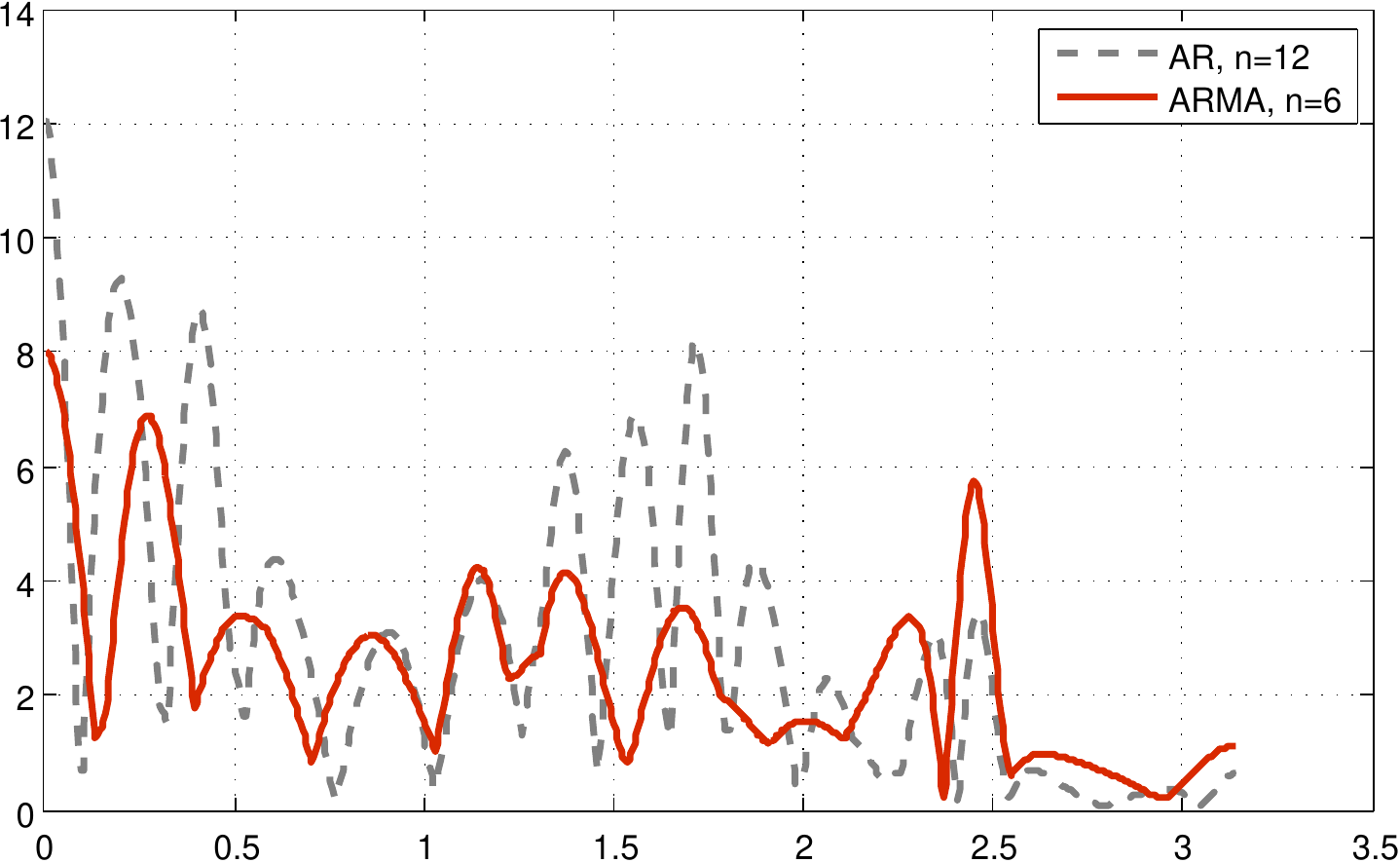}
	\caption{The norm of the approximation error for a bilateral AR of order $12$ for $N=64$ and a bilateral ARMA of order $6$ for $N=32$}
	\label{fig:ARMA_n6_errorNorm} 
\end{figure}  

Given $C$ and $\gammab$, we apply the procedure in this section to determine a pair $(\Pb,\Qb)$ of order $n=6$. 
For comparison we also compute an bilateral AR approximation with $n=12$ fixing $\Pb=\Ib$. As illustrated in Fig.~\ref{fig:ARMA_n6_errorNorm}, the bilateral ARMA model of order $n=6$ computed with $N=32$ outperforms the bilateral AR model of order $n=12$ with $N=64$.

The results of Section~\ref{bilateralsec} can  also be generalized to the multivariate case along the lines described in \cite{Carlietal}.

\section{Application to image processing}\label{imagesec}

In \cite{Carli-FPP} the circulant maximum-entropy solution has been used to model spatially stationary images ({\em textures}) \cite{Soatto}  in terms of (vector-valued) stationary periodic processes. The image could be thought of as an $m\times M$ matrix of pixels where the columns form a $m$-dimensional reciprocal process $\{y(t)\}$, which can extended to
a periodic process with period $M>N$ outside the interval $[0,N]$; see Figure~\ref{imagefig}. 
\begin{figure}[htb]
\begin{center}
\includegraphics[totalheight=4.5cm]{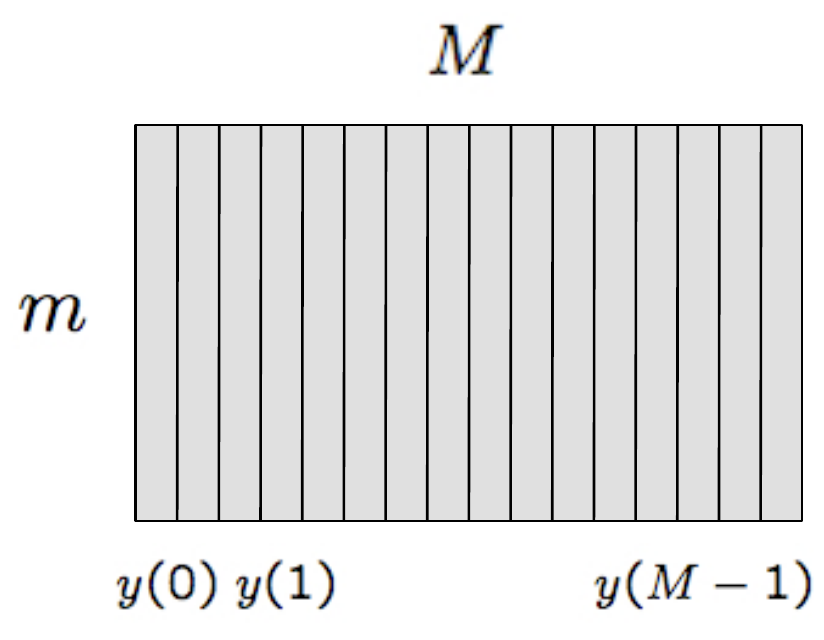}
\caption{An image modeled as a reciprocal vector process}
\label{imagefig}
\end{center} 
\end{figure}
 This imposes the constraint $C_{M-k}=C_k\Tr$ on the covariance lags  $C_k:=E\{ y(t+k)y(t)\Tr\}$, 
leading to a circulant Toeplitz matrix. The problem considered in \cite{Carli-FPP} is to model  the process $\{y(t)\}$ given (estimated)  $C_0,C_1,\dots,C_n$, where $n<N$ with an efficient low-dimensional model. This is precisely a problem of the type considered in Section~\ref{multivariatesec}. 

Solving the corresponding circulant maximum-entropy problem (with $\Pb=\Ib$), $n=1$, $m=125$ and $N=88$, Carli, Ferrante, Pavon and  Picci  \cite{Carli-FPP} derived a bilateral model of the images at the bottom row of Figure~\ref{3images} to compress the images in the top row, thereby achieving a compression of  5:1.
\begin{figure}[htb]\begin{center}
\includegraphics[totalheight=4cm]{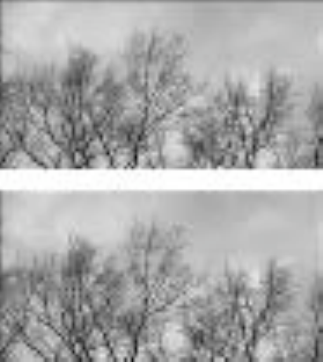}
\qquad\qquad
\includegraphics[totalheight=4cm]{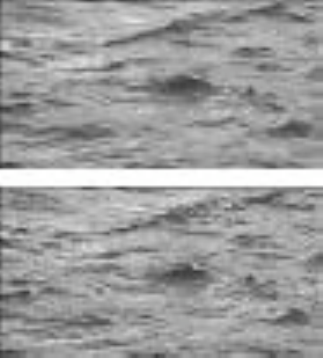}
\qquad\qquad
\includegraphics[totalheight=4cm]{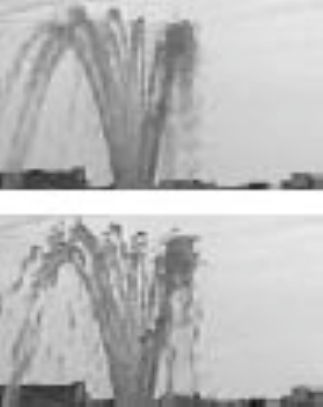}
\caption{Three images modeled by reciprocal processes (original at bottom).}\label{3images}\end{center} 
\end{figure}
 While the compression ratio falls short of competing with current jpeg standards (typically 10:1 for such quality), our approach suggests a new stochastic alternative to image encoding. Indeed the results in Figure~\ref{3images} apply  just the maximum entropy solution of order $n=1$.  Simulations such as those in Figure~\ref{figures5} suggest that much better compression can be made using bilateral ARMA modeling. 
 
 An alternative approach to image compression using  multidimensional covariance extension can be found in the recent paper \cite{RKL}.

% \section{Conclusions and open problems} \label{conclusionsec}  

%\appendix
%
%\section{\empty}

\end{document}